\documentclass[a4paper]{amsart}
\usepackage[utf8]{inputenc}
\usepackage[english]{babel}
\usepackage{fancyhdr}
\usepackage{amsmath}
\usepackage{amsfonts,amstext}
\usepackage{amssymb,amsthm,amscd,amsxtra,wasysym,graphicx}
\usepackage{grffile}
\usepackage{mathrsfs,esint,comment}
\usepackage{tikz-cd}
\usepackage{enumitem,mathtools,dsfont}
\usepackage{palatino,mathpazo}
\usepackage{xcolor}

\def\Bibtex{{\rm B\kern-.05em{\sc i\kern-.025em b}\kern-0.08em T\kern-.1667em\lower.7ex\hbox{E}\kern-.125emX}}
\usepackage[breaklinks=true,bookmarks=false,pagebackref]{hyperref}
\hypersetup{linktocpage=true,
	unicode=false,          
	pdftoolbar=true,       
	pdfmenubar=true,       
	pdffitwindow=false,     
	pdfstartview={FitH}, 
	pdftitle={Monge-Amp\`ere equations},   
	pdfauthor={Quang-Tuan Dang},   
	colorlinks=true,  
	linkcolor=purple,         
	citecolor=blue,        
	filecolor=green,
	urlcolor=blue}          

\usepackage{cleveref}
\usepackage{indentfirst}
\theoremstyle{plain}
\newtheorem{theorem}{Theorem}
\newtheorem{lemma}[theorem]{Lemma}
\newtheorem{corollary}[theorem]{Corollary}

\newtheorem{proposition}[theorem]{Proposition}
\newtheorem{example}[theorem]{Example}

\theoremstyle{definition}
\newtheorem{definition}[theorem]{Definition}

\theoremstyle{plain}
\newtheorem{bigthmm}{Theorem}

\numberwithin{theorem}{section}
\numberwithin{equation}{section}
 
\DeclareMathOperator \PSH {{\rm PSH}}

\def\om{\omega}

\def\f{\varphi}
\def\dc{dd^c}

\begin{document}
	
	\title[Uniform estimates for complex Monge-Amp\`ere equations]{Uniform estimates for complex Monge-Amp\`ere equations II: Hermitian manifolds}
	\author{Quang-Tuan Dang, Lei Zhang and Bin Zhou}
		\address{Yau Mathematical Sciences Center, Tsinghua University, Beijing 100084}
		\email{dangquangtuan10@gmail.com $\&$ dangqt@mail.tsinghua.edu.cn}
        \address{School of Mathematics and Statistics, Jiangsu Normal University, Xuzhou, 221009}
        \email{leizizhang1992@163.com}
        \address{School of Mathematical Sciences, Peking University, Beijing 100871}
        \email{bzhou@pku.edu.cn}
	\date{\today}
	\keywords{Complex Monge-Amp\`ere equations, uniform estimates}
	\subjclass[2020]{32U20, 32W20, 32U05}

	\begin{abstract}
We establish uniform estimates for degenerate complex Monge--Amp\`ere
equations on compact Hermitian manifolds under weak Orlicz integrability
assumptions on the densities. Our estimates apply uniformly to Hermitian
metrics and to semipositive uniformly non-collapsing reference forms.
 For arbitrary big reference forms, we obtain uniform relative
 $L^\infty$-estimates with respect to the extremal function with minimal singularities, and establish continuity of the resulting solutions on the regular locus. 
 As an application, we prove the existence of solutions with minimal singularities to degenerate complex Monge--Amp\`ere equations associated with big forms. 
The proofs combine local Orlicz estimates with
pluripotential-theoretic envelope techniques and uniform control of the
Monge--Amp\`ere masses, overcoming the lack of invariance of the total
Monge--Amp\`ere mass in the Hermitian setting.
	\end{abstract}
	\maketitle
	\tableofcontents
\section{Introduction}

Complex Monge--Amp\`ere equations are among the fundamental nonlinear
equations in complex geometry and pluripotential theory.  On compact
K\"ahler manifolds, their systematic study goes back to Yau's solution
of the Calabi conjecture~\cite{yau1978ricci}.  Besides their central
role in the construction of canonical metrics, these equations have
motivated the development of a powerful pluripotential theory which
allows one to treat degenerate right-hand sides and weak solutions
under very mild regularity assumptions.

Extending this theory beyond the
K\"ahler setting is considerably more delicate.  
Several fundamental features of K\"ahler theory are no longer available. Most notably, the total Monge--Amp\`ere mass may depend on potentials, and integration by parts and
comparison arguments produce additional terms involving the torsion
of the background metric.  Thus even though the complex
Monge--Amp\`ere operator remains locally of the same nature, its
global behavior is considerably more subtle.  A basic problem is
therefore to understand to what extent the a priori estimates and
pluripotential methods developed in the K\"ahler setting continue to
hold on general compact Hermitian manifolds.


The smooth Hermitian equation was first studied by
Cherrier~\cite{cherrier1987equation}, who adapted the
Aubin--Yau method~\cite{aubin1978equation,yau1978ricci}, and then by
Hanani~\cite{hanani1996equation,hanani1996generalisation}.
After estimates of Guan--Li~\cite{guan2010complex} under additional
differential conditions, Tosatti--Weinkove
\cite{tosatti2010complex,tosatti2010estimates} established the required
uniform estimate and solved the equation on arbitrary compact
Hermitian manifolds. For applications of the Chern--Ricci flow, we refer interested readers to
\cite{tosatti2013chern,tosatti2015collapsing,tosatti2015evolution}.

Parallel to these developments, a pluripotential theory for complex
Monge--Amp\`ere equations on compact Hermitian manifolds has gradually
been developed. Fundamental contributions are due to Dinew,
Ko{\l}odziej, Nguyen and others; see
\cite{kolodziej1998complex,dinew2012pluripotential,kolodziej2015weak,
kolodziej2019stability,kolodziej2018holder,nguyen2016complex,kolodziej2021continuous,kolodziej2023weak,kolodziej2023dirichlet}
and surveys
\cite{dinew2016pluripotential,dinew2019lectures}.
These works established Hermitian analogs of several fundamental
tools from K\"ahler pluripotential theory, including capacity
estimates, comparison and domination principles, existence of weak
solutions, stability estimates, and regularity results.

A further development was initiated by Guedj--Lu
\cite{guedj2022quasi2,guedj2021quasi}, who introduced a systematic
envelope approach to degenerate complex Monge--Amp\`ere equations on
compact Hermitian manifolds. In particular, their framework allows
one to deal with semipositive reference forms for which, in contrast
to the K\"ahler case, the total Monge--Amp\`ere mass is not fixed.
More recently, building on the pluripotential theory of big
cohomology classes in the K\"ahler setting
\cite{boucksom2010monge}, Boucksom--Guedj--Lu
\cite{BoucksomGuedjLu2025-volume} developed a theory of
non-pluripolar products on compact Hermitian manifolds and studied
degenerate complex Monge--Amp\`ere equations associated with big
reference forms. For further developments, we refer the reader to~\cite{LWZ23-degenerate,ALM24-degenerate,ALEHYANELuSalouf26-Hermitian,Li26-}.

Despite these advances, uniform estimates become particularly
delicate when both the reference forms and the right-hand sides are
allowed to degenerate. On the one hand, the reference form may
degenerate from a Hermitian metric to a semipositive uniformly
non-collapsing form, or more generally to a big form. On the other
hand, the density of the Monge--Amp\`ere measure may satisfy only
weak integrability assumptions. While the classical $L^p$ theory,
$p>1$, provides powerful a priori estimates, it does not cover the
weaker, nearly borderline integrability regimes considered here.

The main purpose of this paper is to establish uniform
$L^\infty$-type estimates in this setting under suitable Orlicz
integrability assumptions on the density. Our approach combines two
complementary ideas: the auxiliary-function method of
Guo--Phong~\cite{Guo-Phong2024-estimates} and the
pluripotential-theoretic envelope method of
Guedj--Lu~\cite{guedj2022quasi2,guedj2021quasi}.
The first ingredient of our argument is a local uniform estimate for
the complex Monge--Amp\`ere equation under an Orlicz integrability
condition. We then globalize this estimate by means of suitable
auxiliary Monge--Amp\`ere equations, plurisubharmonic envelopes, and
domination arguments. This local-to-global approach is sufficiently
flexible to treat, within the same framework, Hermitian reference
metrics, semipositive uniformly non-collapsing forms, and big
reference forms.

The present work can also be viewed as a Hermitian counterpart of
our previous work~\cite{DangZhangZhou26-estimate-big}. The main new
difficulty in the Hermitian setting is that the Monge--Amp\`ere mass
is no longer preserved. Consequently,
uniform estimates must take into account the possible variation and
even degeneration of the total Monge--Amp\`ere mass.


\vskip 10pt

Before stating our main results, we fix some notation.  Let
$(X,\omega_X)$ be a compact Hermitian manifold of complex dimension
$n$, and let $dV$ be a fixed smooth positive volume form on $X$.   
We let $d$, $d^c$ denote the real differential operators on $X$ defined by $d:=\partial+\Bar{\partial}$, $d^c:=\frac{i}{2}(\Bar{\partial}-\partial)$ so that $\dc=i\partial\Bar{\partial}$.
Given a smooth real $(1,1)$-form $\theta$ on $X$, we denote by
$\PSH(X,\theta)$ the space of $\theta$-plurisubharmonic functions.

\medskip

The first ingredient in our approach is a local uniform estimate for
the Dirichlet problem on a bounded hyperconvex domain
$\Omega\Subset\mathbb C^n$.  More precisely, let
$\mu=f\,dV$ be a positive measure on $\Omega$ with total mass at most
one, and assume that $f\in L^w(\Omega,dV)$, where $w$ satisfies
Condition~\hyperlink{conditionK}{(K)}.  Then the solution of
\[
        (dd^c u)^n=\mu,
        \qquad
        u|_{\partial\Omega}=0,
\]
admits a uniform $L^\infty$ bound depending only on the corresponding
Orlicz data.  The proof combines Cegrell's pluripotential theory
\cite{cegrell1998pluricomplex,cegrell2004general},
plurisubharmonic envelopes, and an auxiliary-function argument.
This local estimate constitutes the main analytic input for the
global estimates developed below.


Our first result is a uniform oscillation estimate for bounded
solutions on an arbitrary compact Hermitian manifold.  It can be
viewed as an Orlicz version of the fundamental $L^\infty$
estimate in the Hermitian pluripotential theory.
\begin{bigthmm}
    Let $(X,\omega_X)$ be an $n$-dimensional compact Hermitian manifold.  Assume that $\mu= fdV$ is a probability measure on $X$, and $0\leq f\in L^w(dV)$ and $f>0$ almost everywhere, where $w$ satisfies \hyperlink{conditionK}{Condition (K)}. Then any solution $\varphi\in \PSH(X,\omega_X)\cap L^\infty(X)$ to  
    \[(\omega_X+\dc\varphi)^n=c\mu,\quad\sup_X\varphi=0 \] where $c>0$, satisfies $Osc(\varphi)\leq C$, for some uniform constant $C>0$ depending only on
$(X,\omega_X)$, $dV$, $n$, $w$, and an upper bound for $\|f\|_{1/n}^{-1}$, $\|f\|_w$.
\end{bigthmm}




The proof is reduced to the local estimate above by constructing
bounded global subsolutions from solutions of local
Monge--Amp\`ere equations.

\medskip

We next turn to degenerate reference forms.  Here, a phenomenon
specific to the Hermitian setting becomes essential: if $\theta$ is
a smooth semipositive $(1,1)$-form, the total Monge--Amp\`ere mass
\(
        \int_X(\theta+dd^c u)^n
\)
may depend on the potential $u$, and hence is not determined solely
by the $dd^c$-class of $\theta$.  Following
Guedj--Lu~\cite{guedj2022quasi2,guedj2021quasi}, we consider the
lower Monge--Amp\`ere volume
\[
 v_-(\theta)
 :=
 \inf\left\{
      \int_X(\theta+dd^c u)^n\,;\;
      u\in\PSH(X,\theta)\cap L^\infty(X)
     \right\}.
\]
The condition $v_-(\theta)>0$ ensures that the Monge--Amp\`ere
masses of bounded $\theta$-psh potentials are uniformly bounded
away from zero.  It thus provides the quantitative non-collapsing
condition needed to compensate for the lack of 
control of the total mass in the Hermitian setting.

This leads to our second main result.
 \begin{bigthmm}\label{thm: B}
    Let $(X,\omega_X)$ be an $n$-dimensional compact Hermitian manifold.  Assume that $\mu= fdV$ is a probability measure on $X$, and $0\leq f\in L^w(dV)$ and $f>0$ almost everywhere, where $w$ satisfies \hyperlink{conditionK}{Condition (K)}. Let $\theta$ be a semipositive (1,1) form on $X$ such that $v_-(\theta)>0$. Assume that $\theta\le A\omega_X $ for a fixed constant $A>0$.
    Let $\lambda\geq 0$.
    Then there exist $\varphi\in \PSH(X,\theta)\cap L^\infty(X)$ and $c>0$ which solve
    \[(\theta+\dc\varphi)^n=ce^{\lambda\varphi}\mu,\quad\sup_X\varphi=0. \] Moreover, $Osc(\varphi)\leq C$ for some uniform constant $C$ which depends only on $X$, $\omega_X$, $dV$, $n$, $w$, $\theta$, and an upper bound for $v_-(\theta)^{-1}$, $\|f\|_{1/n}^{-1}$, $\|f\|_w$. The constant $C$ can be chosen independently of $\lambda\geq0$.
\end{bigthmm}



The proof proceeds by approximating $\theta$ by Hermitian forms and
establishing estimates that remain uniform as the approximating
metrics degenerate.  The condition $v_-(\theta)>0$ provides uniform
control of the normalizing constants, while an auxiliary
Monge--Amp\`ere equation, combined with a suitable plurisubharmonic
envelope construction, yields the required uniform oscillation
estimate.

\medskip 
We finally turn to big classes.  Recall that a smooth real
$(1,1)$-form $\theta$ is called \emph{big} if there exists
$\rho_0\in\PSH(X,\theta)$ with analytic singularities such that
\[
                 \theta+dd^c\rho_0
                 \geq\varepsilon_0\omega_X
\]
in the sense of currents, for some $\varepsilon_0>0$. Set $\Omega:=X\setminus\{\rho_0=-\infty\}$. We denote by
\[
        V_\theta
        :=\sup\{u\in\PSH(X,\theta)\,;\ u\leq0\}
\]
the $\theta$-psh function with minimal singularities.


In contrast to the previous settings, solutions associated with a
big reference form need not be globally bounded.  The natural
quantity to estimate is therefore their singularity in relation to
$V_\theta$.  More precisely, one seeks solutions $\varphi$ with minimal
singularities and a uniform bound for
\(
        |\varphi-V_\theta|.
\)

Our third main result establishes such a relative estimate and, at
the same time, yields the existence of solutions to the corresponding
degenerate complex Monge--Amp\`ere equations.
\begin{bigthmm}\label{thm: C}
  Let $(X,\omega_X)$ be an $n$-dimensional compact Hermitian manifold,
and let $\theta$ be a smooth real big $(1,1)$-form satisfying
$\theta\leq A\omega_X$.
Assume that $\mu=f\,dV$ is a probability measure, where
$0\leq f\in L^w(dV)$, and $w$ satisfies
\hyperlink{conditionK}{Condition~(K)} and \hyperlink{conditionI}{Condition~(I)}.
Then, for every $\lambda\geq0$, there exist
$\varphi\in\PSH(X,\theta)$ and $c>0$ such that \(
                         \varphi=V_\theta+O(1) and
\)
\[
        (\theta+dd^c\varphi)^n
        =c e^{\lambda\varphi}\mu,
        \qquad
        \sup_X\varphi=0.
\]
More precisely,
\(
                         |\varphi-V_\theta|\leq C\),
 for some uniform constant $C$ which depends on  $(X,\omega_X)$, {$\theta$}, $n$, $dV$, $w$, $\lambda$, and an upper bound for $\|f\|_{w}$ and $\|f\|_{1/n}^{-1}$.

 Moreover, $\varphi$ is continuous on $\Omega$.
\end{bigthmm}
The proof relies on the pluripotential theory of big forms on
compact Hermitian manifolds, together with the uniform estimates
established in the preceding sections.  We approximate $V_\theta$
by quasi-psh functions with solvable singularities and resolve these
singularities by suitable modifications of $X$.  The resulting
Monge--Amp\`ere equations can then be studied on the modifications
with respect to genuine Hermitian reference metrics.  The main point
is to obtain estimates that are uniform along the approximation
process, so that they persist when passing to the limit.

\subsection*{Acknowledgments.} 
Q.-T. Dang is supported by  the Shuimu Scholar program of Tsinghua University. L. Zhang is supported by Postdoctoral Fellowship Program of China  GZC20240867.
B. Zhou is partially supported by  National Key R$\&$D Program of China 2023YFA009900 and NSFC  Grant 12271008.

\subsection*{AI Disclosure.} The authors used ChatGPT solely to assist with English-language editing and to identify potential gaps, and inconsistencies in the proofs. The authors take full responsibility for all mathematical arguments, claims, and any errors that remain.
\section{Preliminaries}

\subsection{Orlicz space}
We recall some background about Orlicz space, which can be found in~\cite{rao2002applications}.
\begin{definition}
    A {\em Young} function is a convex increasing lower semicontinuous function $w:\mathbb{R}^+\rightarrow\mathbb{R}^+$ such that $w(0)=0$, $\displaystyle\lim_{t\to 0}\frac{w(t)}{t}=0$, $\displaystyle\lim_{t\to+\infty}\frac{w(t)}{t}=+\infty$ 
    Its {conjugate function} is the Legendre transform of $w(\cdot)$, i.e.,
    \[ w^*(t)=\sup_{s\geq 0}(st-w(s)).\]
     We say that $w$ satisfies the $\Delta_2$-condition if there is $A>0$ such that \begin{equation}\label{eq: Delta_2}
         w(2t)\leq Aw(t),\forall\, t>t_0.
    \end{equation} 
\end{definition}

\begin{definition}
    Let $\mu$ be a positive measure on $X$ and let $w$ be a Young function. The {\em Orlicz space} $L^w(X,d\mu)$ is the set of all measurable functions $f$ on $X$ such that
    \[\int_X w(\varepsilon |f|)d\mu<+\infty \]
    for some $\varepsilon>0$. The {\em Luxemburg norm} of $f\in L^w(\mu)$ is 
    \[ \|f\|_{w,\mu}:=\inf\left\{r>0:\int_X w(r^{-1}|f|)d\mu\leq 1 \right\}.\]
    When $\mu=dV$ is a smooth volume form on $X$, we simply write $\|f\|_w$ and $\|f\|_p$ if $w(t)=t^p$ for $p\geq 1$. 
\end{definition}

If $w$ satisfies the $\Delta_2$-condition, then the uniform boundedness of the Luxemburg norm is equivalent to the uniform boundedness of $\int_Xw(|f|)d\mu$. 

Assume that $\|f\|_{w,\mu}\leq C$.
For every $\varepsilon>0$, the definition of the Luxemburg norm gives \[ \int_X w\left(\frac{|f|}{C+\varepsilon}\right)d\mu \leq 1. \] The $\Delta_2$-condition implies that, for every fixed $\lambda>1$, there exists $A_\lambda>0$ such that \[ w(\lambda t)\leq A_\lambda w(t), \] up to an inessential additive constant if the $\Delta_2$-condition is assumed only for large $t$. Therefore, \[ \begin{split} \int_X w(|f|)\,d\mu &= \int_X w\left( (C+\varepsilon) \frac{|f|}{C+\varepsilon} \right)d\mu \\ &\leq A_{C+\varepsilon} \int_X w\left(\frac{|f|}{C+\varepsilon}\right)d\mu \leq A_{C+\varepsilon}. \end{split} \] Conversely, assume that \( \int_X w(|f|)\,d\mu\leq C\). 
When $C\leq1$, this follows immediately from the definition of the Luxemburg norm. Otherwise, $C>1$, the convexity of $w$, together with $w(0)=0$, gives \[ w\left(\frac{t}{C}\right) \leq \frac{1}{C}w(t). \] Therefore, \[ \int_X w\left(\frac{|f|}{C}\right)d\mu \leq \frac1C\int_X w(|f|)\,d\mu \leq1, \] and hence $\|f\|_{w,\mu}\leq C$.

\begin{definition} \label{conditionK}
   \hypertarget{conditionK}{We say that a Young function $w$ satisfies {\em Condition (K)} if 
   $w=w_h$ satisfies the $\Delta_2$-condition and
   there exists an increasing function $h:(1,+\infty)\to(1,+\infty)$ such that
    \[w(t)=w_h(t)\asymp_{+\infty}t h(\log(e+t)), \qquad\int_1^{+\infty}h^{-1/n}(t) dt<+\infty.\]}
\end{definition}
\begin{example} We have several examples of weights satisfying \hyperlink{conditionK}{Condition (K)}:
 \begin{itemize}
        \item $w(t)=t^p$, for some $p>1$;
        \item $w(t)=t(\log(e+t))^{p}$, for some $p>n$;
        \item $w(t)=t(\log(e+t))^n(\log(e+\log(1+t)))^p$, for some $p>n$.
    \end{itemize}
\end{example}

\begin{lemma}\label{lem: weak-weight}
    Let $h:[1,+\infty)\to(0,+\infty)$ be an increasing function such
that
\begin{equation}\label{eq:int-Phi}
    \int_1^{+\infty}\frac{dt}{h(t)^{1/n}}<+\infty.
\end{equation}
We assume, moreover, that the weight
\(w_h(t)\asymp_{+\infty}
    t\,h\bigl(\log(e+t)\bigr)\),
satisfies the $\Delta_2$-condition.
    Then there exists an  increasing function $\widetilde h:(0,+\infty)\to(1,+\infty)$ such that
    \[\lim_{t\to+\infty}\frac{\widetilde h(t)}{ h(t)}=0, \qquad\int_1^{+\infty}\widetilde h^{-1/n}(t) dt<+\infty.\]
    Moreover, the associated weight $w_{\widetilde h}$ satisfies the $\Delta_2$-condition. 
\end{lemma}


\begin{proof}
Set
\[
    G(t):=\int_t^{+\infty}\frac{dt}{h(t)^{1/n}},
\]
so that $G(t)\to0$ as $t\to+\infty$. Then $g(t):=G(t)^{-1/2}$ goes to $+\infty$ as $t\to+\infty$. A computation gives
\[\int_1^{+\infty}\frac{g(t)dt}{h^{1/n}(t)}=-\int_1^{+\infty}G'(t)G(t)^{-1/2}dt=2G(1)^{1/2}<\infty. \]
Define $ h_0(t):=
    \frac{h(t)}{g(t)^n}$.
Then $$ \int_1^{+\infty}\frac{dt}{h_0(t)^{1/n}}
    <+\infty\;\;\text{and}\;\;\frac{h(t)}{h_0(t)}
    =
    g(t)^n
    \xrightarrow{t\to+\infty}+\infty.$$
The function $h_0$ may not be increasing. 
Inspired by~\cite{Sinnamon2007-monotonicity}, we replace $h_0$ by its {\em least increasing majorant}
\begin{equation}\label{eq:Psi-def}
   \widetilde h (t):=
    \sup_{1\leq s\leq t}h_0(s).
\end{equation}
Clearly, $\widetilde h$ is increasing and $\widetilde h\geq h_0$. 
 It follows that
$
    \int_1^{+\infty}\frac{dt}{\widetilde h(t)^{1/n}}
    <+\infty$.
We next show that $\widetilde h$ is asymptotically strictly weaker than
$h$, namely
\begin{equation}\label{eq:Psi-o-Phi}
   \frac{\widetilde h(t)}{h(t)}
    \longrightarrow0
    \qquad\text{as }t\to+\infty.
\end{equation}
Fix $N>1$. Since $g(t)\to+\infty$, as $t\to+\infty$, there exists
$T_N>1$ such that
\(g(t)^n\geq N \), for
    $ t\geq T_N$.
Thus, $h_0(t)\leq\frac{h(t)}{N}$, $t\geq T_N$.
Since $h$ is increasing, for $t\geq T_N$ we obtain
\[
\begin{split}
    \widetilde h(t)
    =
    \sup_{1\leq s\leq t}h_0(s)
    \leq
    \max\left\{
        C_N,\frac{h(t)}{N}
    \right\},
\end{split}
\]
where
\(
    C_N:=\sup_{1\leq t\leq T_N}h_0(t)<+\infty.
\)
Since $h(t)\to+\infty$, for
$t$ sufficiently large,
    $C_N\leq\frac{h(t)}{N}$,
and therefore
   $ \widetilde h(t)\leq\frac{h(t)}{N}.$ Since $N>1$ is arbitrary, \eqref{eq:Psi-o-Phi} follows.

Since $w_h$ satisfies the
$\Delta_2$-condition,
there exists $C>1$ such that, for all sufficiently large $t$,
\begin{equation}\label{eq:Phi-shift}
    h(t+\log 2)\leq Ch(t).
\end{equation}
Since $g$ is increasing, we obtain
\begin{equation}\label{eq:Psi0-shift}
\begin{split}
    h_0(t+\log2)
    =
    \frac{h(t+\log2)}
         {g(t+\log2)^n}
    \leq
    C\frac{h(t)}{g(t)^n}
    =
    Ch_0(t).
\end{split}
\end{equation}
The same estimate holds for the
increasing majorant $\widetilde h
$:
\begin{equation}\label{eq:Psi-shift}
    \widetilde h(t+\log2)\leq C_1\widetilde h(t)
\end{equation}
for all sufficiently large $t$. In particular, the associated weight
\(
    w_{\widetilde h}(t)
\)
also satisfies the $\Delta_2$-condition. 
Iterating this inequality yields
constants $C_2$, $\beta>0$ such that
\begin{equation}\label{eq:Psi-general-shift}
    \widetilde h(t+r)
    \leq C_2 e^{\beta r}\widetilde h(t),
    \qquad r\geq0,
\end{equation}
for all sufficiently large $t$.
\end{proof}

\begin{lemma}\label{lem: weak-density} Let $h,\widetilde h$ be the functions in Lemma \ref{lem: weak-weight}.
    For every $a>0$, there exists a
positive increasing function $k$ such that
\begin{equation}\label{eq:master-k}
    w_{\widetilde h}\bigl(tk(s)\bigr)
    \leq
    C_a\left(1+w_h(t)+e^{as}\right),
    \qquad t\geq0,\quad s\geq2,
\end{equation} where $C_a>0$ depends only on $a>0$.
\end{lemma}
\begin{proof}
Recall that
\[
    w_h(t)\asymp_\infty
    t\,h\bigl(\log(e+t)\bigr),
    \qquad
    w_{\widetilde h}(t)\asymp_\infty
    t\,\widetilde h\bigl(\log(e+t)\bigr),
\]
and set
\(
    R(r):=\frac{h(r)}{\widetilde h(r)}.
\)
By construction,
\(R(r)\rightarrow+\infty\)
    as  $r\to+\infty$.
We consider the {\em 
greatest increasing minorant} of $R$:
\[
    R_*(r):=\inf_{s\geq r}R(s).
\]
Then $R_*$ is increasing,
    $R_*(r)\leq R(r)$, and $
    R_*(r)\rightarrow+\infty$ as $r\to+\infty$.

{The estimate \eqref{eq:Psi-general-shift} implies, for some $C,\beta>0$,}
\begin{equation}\label{eq:whtilde-polynomial}
    w_{\widetilde h}(r)
    \leq
    C\left(1+r^{1+\beta}\right),
    \qquad r\geq0.
\end{equation}
Choose $\eta>0$ small so that
\begin{equation}\label{eq:eta-master}
    (1+\beta)\eta<1.
\end{equation}
{Fix $K>\max\{1,2(1+\beta)/a\}$ and define}
\begin{equation}\label{eq:k-master}
    k(s):=
    \min\left\{
        R_*\left(\frac{s}{K}\right)^\eta,
        e^{\sqrt{s}}
    \right\}
\end{equation}
for all sufficiently large $s$, extending $k$ to a positive
nondecreasing function on $[1,+\infty)$. Then
\(
    k(s)\rightarrow+\infty,
\) as $s\to+\infty$
and
\begin{equation}\label{eq:k-master-R}
    k(Kr)
    \leq
    R_*(r)^\eta
    \leq
    R(r)^\eta.
\end{equation}
Moreover,
\begin{equation}\label{eq:k-master-subexp}
    k(s)\leq e^{\sqrt{s}},
    \qquad
    \log k(s)=o(s).
\end{equation}
We now prove \eqref{eq:master-k}. Fix $t\geq0$ and $s\geq2$, and set
\(
    L:=\log(e+t).
\)
We distinguish two cases.

\medskip
\noindent
\textbf{Case 1: $s\leq KL$.}
Since $k$ is nondecreasing, \eqref{eq:k-master-R} yields
\[
    k(s)
    \leq
    k(KL)
    \leq
    R(L)^\eta.
\]
Since $t\leq e^L$, we obtain
\[
    \log\bigl(e+tk(s)\bigr)
    \leq
    L+\eta\log R(L)+C.
\]
{Using \eqref{eq:Psi-general-shift}, we get}
\[
\begin{split}
    \widetilde h\bigl(\log(e+tk(s))\bigr)
    \leq
    C\widetilde h
    \bigl(L+\eta\log R(L)+C\bigr)\leq
    C R(L)^{\beta\eta}\widetilde h(L).
\end{split}
\]
Consequently,
\[
\begin{split}
    w_{\widetilde h}\bigl(tk(s)\bigr)
    \lesssim
    tk(s)\,
    \widetilde h\bigl(\log(e+tk(s))\bigr)
    \lesssim
    tR(L)^{(1+\beta)\eta}\widetilde h(L).
\end{split}
\]
Since $(1+\beta)\eta<1$ and $R(L)\geq1$ for $L$ sufficiently large,
\[
    R(L)^{(1+\beta)\eta}
    \leq R(L).
\]
Thus
\begin{equation}\label{eq:master-case1}
    w_{\widetilde h}\bigl(tk(s)\bigr)
    \leq
    C\bigl(1+w_h(t)\bigr)
\end{equation}
whenever $s\leq KL$.

\medskip
\noindent
\textbf{Case 2: $s>KL$.}
{Since $t\leq e^L<e^{\frac{s}{K}}$, \eqref{eq:k-master-subexp} gives
$tk(s)\leq e^{\frac{s}{K}+\sqrt{s}}$. Our choice of $K$ ensures that
$(1+\beta)(\frac{s}{K}+\sqrt{s})\leq as+\frac{(1+\beta)^2}{2a}$.
Thus \eqref{eq:whtilde-polynomial} yields
\begin{equation}\label{eq:master-case2}
\begin{aligned}
    w_{\widetilde h}\bigl(tk(s)\bigr)
    &\leq C\left(1+e^{(1+\beta)(\frac{s}{K}+\sqrt{s})}\right)\\
    &\leq C_a(1+e^{as}).
\end{aligned}
\end{equation}
Combining \eqref{eq:master-case1} and \eqref{eq:master-case2} proves \eqref{eq:master-k}.}
\end{proof}

\noindent\textbf{Condition (I).} \label{conditionI} \hypertarget{conditionI}
We further say that the Young function $w_h$ satisfies \hyperlink{conditionI}{Condition (I)} if the function $k$ in Lemma~\ref{lem: weak-density} satisfies the following integrability condition
\begin{equation*}
    \int_1^\infty\frac{dt}{k(t)^{1/n}}<\infty.
\end{equation*}

\begin{example}
     We give several examples. {Fix $a>0$, as in Lemma~\ref{lem: weak-density}. In the last two examples, the displayed formulas for $w_h$ and $w_{\widetilde h}$ are understood for sufficiently large $t$, with convex modifications near $0$ so that they satisfy Definition~2.1.}   \begin{itemize}
        \item If $h=e^{(p-1) t}$, $w_h(t)=t^{p}$, with $p>1$, then we can pick 
        \[\widetilde{h}=e^{(q-1)t}, \ \ w_{\widetilde h}(t)=t^q,\quad1<q<p. \]
Then $k$ can be chosen so that ${k(t)=e^{bt}}$ with ${0<b\leq \frac{a(p-q)}{pq}}$.
        \item If $h=t^p$, $w_h(t)=t[\log(e+t)]^p$, with $p>2n$, then we can pick 
        \[\widetilde{h}(t)=t^q, \ \ w_{\widetilde h}(t)=t[\log(e+t)]^q,\quad n<q<p-n. \]
Then $k$ can be chosen so that ${k(t)=t^{b}}$ with ${n<b<p-q}$.
    \end{itemize}
\end{example}

\subsection{Cegrell Classes}
We fix a bounded hyperconvex domain $\Omega\subset\mathbb C^n$, i.e.,
there exists a continuous negative psh exhaustion function
$\rho:\Omega\to[-1,0)$ such that
\(
        \{\rho<c\}\Subset\Omega,\; \forall\,c<0\). 
        We let $\PSH(\Omega)$ denote the set of plurisubharmonic (psh for short) functions on $\Omega$.


\begin{lemma}\label{lem: max-princ-local}
    Assume $u\leq v$ are bounded psh functions on $\Omega$. Then \[\mathbf{1}_{\{u=v\}}(\dc u)^n\leq  \mathbf{1}_{\{u=v\}}(\dc v)^n.\]
\end{lemma}
\begin{proof}
    See~\cite[Corollary 3.28]{guedj2017degenerate}.
\end{proof}

We recall some essential concepts of the Cegrell classes 
\cite{cegrell1998pluricomplex,cegrell2004general}. Let $\mathcal{E}_0(\Omega)$ denote the set of bounded plurisubharmonic functions $u$ in $\Omega$ such that $\lim_{z\to\zeta}u(z)=0$ for every $\zeta\in\partial\Omega$, and $\int_\Omega(\dc u)^n<+\infty$.

Let $\mathcal{E}(\Omega)$ denote the set of psh functions $u$ such that for all $z_0\in \Omega$, there exists a neighborhood $V_{z_0}$ of $z_0$ and a decreasing sequence $u_j\in\mathcal{E}_0(\Omega)$ which converges to $u$ in $V_{z_0}$ and
satisfies $\sup_j\int_\Omega(\dc u_j)^n<+\infty$. 

A function $u$ belongs to $\mathcal{F}(\Omega)$ iff there exists a sequence of functions $u_j\in\mathcal{E}_0(\Omega)$ decreasing to $u$ in all
of $\Omega$, which satisfies $\sup_j\int_\Omega (\dc u_j)^n<+\infty$. Let $\mathcal{F}^a(\Omega)$ be the subclass of function $u\in\mathcal{F}(\Omega)$ such that $(\dc u)^n$ vanishes on all pluripolar sets.

The class $\mathcal{E}^p(\Omega)$ (respectively $\mathcal{F}^p(\Omega)$), for $p>0$, is the set of psh functions $u$ for which there
exists a sequence of functions $u_j\in\mathcal{E}_0(\Omega)$ decreasing towards $u$ in all of $\Omega$, and
so that \[\sup_j\int_\Omega(-u_j)^p(\dc u_j)^n<+\infty \]
(respectively, $\sup_j\int_\Omega[1+(-u_j)^p](\dc u_j)^n<+\infty$).

As demonstrated by Cegrell \cite{cegrell1998pluricomplex,cegrell2004general}, the complex Monge–Amp\`ere operator can be
defined for functions belonging to the class $\mathcal{F}(\Omega)$ even when these functions are unbounded.

The Dirichlet problem in $\mathcal{F}^a(\Omega)$ was solved in~\cite[Lemma 5.14]{cegrell2004general}.

\begin{theorem}\label{thm: cegrell}
    Let $\mu$ be a positive measure on $\Omega$. If $\mu(\Omega)<+\infty$ and $\mu$ vanishes
on all pluripolar sets, then there exists a function $u\in\mathcal{F}^a(\Omega)$ such that $u|_{\partial\Omega}=0$ and $(\dc u)^n=\mu$.
\end{theorem}
In \cite[Theorem 5.15]{cegrell2004general}, the comparison principle is established as in the following.
\begin{theorem}\label{thm: local-comparison}
    If $u\in\mathcal{F}^a(\Omega)$ and $v\in\mathcal{E}(\Omega)$ with $(\dc v)^n\geq (\dc u)^n$ then $u\geq v$.
\end{theorem}
\subsection{Plurisubharmonic envelopes}
For a measurable function $f$, the envelope $P(f)$ is defined by
\[P(f)=(\sup\{ u\in\PSH(\Omega):u\leq f\})^*. \]
The study of this object has attracted the interest of several authors over the last
decade (see~\cite{bedford1982new,guedj2019plurisubharmonic}, and the references therein for more
information).
\begin{theorem}\label{thm: envelope-local}
If $h$ is bounded and quasi lower semicontinous, then $P(h)$ is a bounded plurisubharmonic function whose Monge-Amp\`ere measure $(\dc P(h))^n$ is carried
on the contact set  $\{P(h)=h\}$.   
\end{theorem}
\begin{proof}
    See~\cite[Theorem 2.19]{DL25-singularCY}.
\end{proof}
\subsection{Global pluripotential theory}
In this section, we recall some basic results from global
pluripotential theory on compact Hermitian manifolds; see~\cite{dinew2019lectures,BoucksomGuedjLu2025-volume,DL25-singularCY}.

Let $X$ denote a compact Hermitian manifold of complex dimension $n$, equipped with a Hermitian form $\omega_X$. 
	Let \(dV:={\omega_X^n}/n! \)
	denote the volume form associated with $\omega_X$. For any $p\geq 1$, we simply write $\|f\|_p$ for the $L^p(X,dV)$-norm of $f$. Let $\theta$ be a real smooth (1,1) form on $X$. 
\subsubsection{Quasi-plurisubharmonic functions}
	We recall that a function $u\in L^1(X)$ is quasi-plurisubharmonic (quasi-psh for short) if it is locally given as the sum of a smooth function and a plurisubharmonic (psh for short) function.
	\begin{definition}
		A quasi-psh function $\f:X\to[-\infty,+\infty)$ is called {\em $\theta$-plurisubharmonic} ({\em $\theta$-psh} for short) if it satisfies $\theta_\f:=\theta+\dc\f\geq 0$ in the weak sense of currents.
		We let $\PSH(X,\theta)$ denote the set of all $\theta$-psh functions that are not identically $-\infty$.
	\end{definition} The set $\PSH(X,\theta)$ is endowed with the $L^1(X)$-topology. By Hartogs' lemma, the map $\f \mapsto \sup_X \f$ is continuous with respect to this topology. 
The standard compactness theorem for normalized
quasi-psh functions therefore implies that
\[\PSH_0(X,\theta)=
 \{u\in\PSH(X,\theta):\sup_Xu=0\}
\]
is compact in $L^1(X)$.
	We refer the reader to~\cite{demaillycomplex,guedj2017degenerate} for basic properties of $\theta$-psh functions. 

     We denote by $$V_\theta=\sup\{u\in\PSH(X,\theta):u\leq 0 \}$$ a $\theta$-psh function with minimal singularities.

	\begin{definition}
		A function $\rho$ is said to have \begin{itemize}
		    \item 
        {\em divisorial singularities} {if locally on $X$,
\[
\rho=\sum_{\ell=1}^N a_\ell\log|s_\ell|+g,
\]
where $a_\ell>0$, the $s_\ell$ are holomorphic functions that are not identically zero, and $g$ is smooth.}
        \item {\em solvable singularities} if there exists a modification $\pi: Y\to X$ (i.e., a proper
bimeromorphic morphism with $Y$ smooth) such that $\pi^*\rho $ has divisorial singularities.
        \end{itemize}
	\end{definition}
    
	

\begin{theorem}[{\cite[Theorem 3.2]{demailly2004numerical}}]\label{thm: demailly-regularization} Let $u\in\PSH(X,\theta)$. Then there exists a decreasing sequence $u_j\in\PSH(X,\theta+\varepsilon_j\omega_X)$ with solvable singularities such that $u_j\searrow u$, where $\varepsilon_j\searrow 0$ as $j\to+\infty$.  
\end{theorem}
We see that $T=\theta+\dc u$ has solvable singularities iff there exists a modification $\pi: Y\to X$ such that \[\pi^*T=\beta+[E], \] 
where $\beta$ is a real smooth (1,1)-form and $E$ is an effective $\mathbb{R}$-divisor.

\subsubsection{Monge--Amp\`ere measures}
	
	The complex Monge--Amp\`ere measure $(\theta+\dc u)^n$ is well-defined for any $\theta$-psh function $u$, which is locally bounded in an open set $U\subset X$, as follows from the Bedford--Taylor theory: if $\beta=\dc\rho$ is a K\"ahler form such that $\beta>\theta$  in a local open chart $V\subset X$, then $u$ is $\beta$-psh and the positive currents $(\beta+\dc u)^j$ are well-defined for $1\leq j\leq n$. Thus, the {\em complex Monge--Amp\`ere measure}
	\begin{equation*}
	(\theta+\dc u)^n:=\sum_{j=0}^n \binom{n}{j}(\beta+\dc u)^j\wedge (\theta-\beta)^{n-j}
	\end{equation*} is a signed measure on $U$.

    This definition does not depend on the choice of $\beta$  using the same arguments; see, e.g.,~\cite[Definition 2.21]{DL25-singularCY}.
The same expansion also gives continuity when the background varies:
if smooth real $(1,1)$-forms $\omega_j\to\theta$ in $\mathcal{C}^0$ and
uniformly bounded $u_j\in\PSH(X,\omega_j)$ decrease to
$u\in\PSH(X,\theta)$, then
\begin{equation}\label{eq:varying-background-convergence}
 (\omega_j+\dc u_j)^n\longrightarrow(\theta+\dc u)^n
 \quad\text{weakly}.
\end{equation}
Indeed, on each relatively compact coordinate chart, choose a fixed
K\"ahler form $\beta$ that dominates all $\omega_j$ and $\theta$.
The Bedford--Taylor convergence theorem along bounded decreasing functions applies 
 to each $(\beta+\dc u_j)^\ell$, while the remaining smooth factors converge. 

In particular, for any bounded $u\in\PSH(X,\theta)$, taking the convolution with a smoothing kernel, we can locally find a sequence of smooth functions $u_j$ decreasing to $u$ such that $\theta+2^{-j}\omega_X+\dc u_j\geq 0$. Then the convergence result  yields
that \[(\theta+ 2^{-j}\omega_X+\dc u_j)^n\longrightarrow(\theta+\dc u)^n
 \quad\text{weakly}. \]
This shows that the limit $(\theta+\dc u)^n$ is also a positive current.

	We refer to~\cite{dinew2012pluripotential,BoucksomGuedjLu2025-volume,DL25-singularCY} for an adaptation of~\cite{bedford1976dirichlet, bedford1982new} to the Hermitian context.
	

	\subsubsection{Envelopes}
	Recall that a  Borel set $E\subset X$ is {\em (locally) pluripolar} if for each $x\in X$, there exists an open neighborhood $U$ of $x$ and a psh function $u$ on $U$ such that $E\cap U\subset \{u=-\infty\}$. 
	As follows from \cite[Theorem 1.1]{vu2019pluripolar} or \cite[Lemma 2.6]{guedj2022quasi2}, the set $E$ is globally pluripolar, i.e., there exists  $u\in \PSH(X,\omega_X)$ such that $E\subset\{u=-\infty \}$. 

	\begin{definition}
		Given a measurable function $h:X\to\mathbb{R}$, we define the {\em $\theta$-psh envelope} of $h$ by
		\begin{equation*}
		P_\theta(h):=(\sup\{u\in\PSH(X,\theta): u\leq h\;\text{on}\, X \})^*,
		\end{equation*} where the star means that we take the upper semi-continuous regularization.
	\end{definition}
 We note that this definition is equivalent to the one given in \cite[Definition 2.2]{guedj2022quasi2}; cf. \cite[Corollary 2.7]{guedj2022quasi2}.

	We have the following result, established in~\cite[Theorem 2.3]{guedj2022quasi2}.
	\begin{theorem}\label{thm: envelope} Let $\theta$ be a smooth semipositive $(1,1)$ form.
		If $h$ is bounded from below, quasi-lower-semi-continuous, and  $P_\theta(h)<+\infty$, then
		\begin{enumerate}
			\item $P_\theta(h)$ is a bounded $\theta$-psh function;
			\item $P_\theta(h)\leq h$ in $X\setminus P$, for some pluripolar set $P$;
			\item $(\theta+\dc P_\theta(h))^n$ is supported on the contact set $\{P_\theta(h)=h\}$.
		\end{enumerate}
	\end{theorem}
\subsubsection{Uniformly non-collapsing forms}
For our purposes, we also recall a stronger positivity property of the form $\theta$ in the sense of~\cite{guedj2022quasi2,guedj2021quasi}. 
	
    First, we recall the following definition.
\begin{definition}\label{def: mass} Let $\theta$ be a real smooth semipositive $(1,1)$ form. 
 We say $\theta$ is {\em non-collapsing} if there is no bounded $\theta$-plurisubharmonic function $u$ such that $(\theta+\dc u)^n\equiv 0$.
         We say that $\theta$ is {\em uniformly non-collapsing} if $v_-(\theta)>0$, where
		\[v_{-}(\theta):=\inf\left\{\int_X(\theta+\dc\f)^n: \f\in\PSH(X,\theta)\cap L^\infty(X) \right\}.\]
	\end{definition}
We emphasize that when $\theta$ is Hermitian, the infimum in the
definition of $v_-(\theta)$ can be taken over
$\PSH(X,\theta)\cap C^\infty(X)$ due to Demailly's regularization theorem and Bedford--Taylor's convergence results. We refer the interested reader to~\cite[Section 3]{angella2023plurisigned} for several examples of uniformly non-collapsing Hermitian forms.

We have the monotonicity property.
\begin{lemma}[{\cite[Proposition 3.2]{guedj2022quasi2}}]\label{lem: monotonicity-volume}
    Let $\theta_1,\theta_2$ be real smooth semipositive (1,1)-forms such that $\theta_1\leq \theta_2$. Then $v_-(\theta_1)\leq v_-(\theta_2)$. 
\end{lemma}

 We recall the following maximum principle.
	\begin{lemma}[{\cite[Lemma 1.2]{guedj2022quasi2}}]\label{lem: max-princ}
		Let $u,v$ be bounded $\theta$-psh functions such that $u\leq v$. Then
		\begin{equation*}
		\mathbf{1}_{\{u=v\}}(\theta+\dc u)^n\leq\mathbf{1}_{\{u=v\}}(\theta+\dc v)^n.
		\end{equation*}
	\end{lemma}

    \begin{proposition}[{\cite[Lemma 1.9]{guedj2021quasi}}]\label{lem: GL23-19} Fix $\lambda\geq 0$. Let $\theta$ be a real smooth semipositive $(1,1)$ form.
    Fix two smooth
semipositive (1, 1)-forms $\omega_1,\omega_2$ such that $\omega_1\geq \theta$, $\omega_2\geq \theta$.
    \begin{enumerate}
        \item Assume $u,v\in\PSH(X,\theta)\cap L^\infty(X)$. If $\theta_u^n\geq e^{\lambda u}fdV$ and $\theta_v^n\geq e^{\lambda v}gdV$ then
        \[\theta_{\max(u,v)}^n\geq e^{\lambda\max(u,v)}\min(f,g)dV. \]
        \item Assume $u\in\PSH(X,\omega_1)\cap L^\infty(X)$, $v\in\PSH(X,\omega_2)\cap L^\infty(X)$. If
        \[(\omega_1+\dc u)^n\leq e^{\lambda u}fdV\;\text{and}\;(\omega_2+\dc v)^n\leq e^{\lambda v}gdV\] then
        \[\theta_{P_\theta(\min(u,v))}^n\leq e^{\lambda P_\theta(\min(u,v))}\max(f,g)dV. \]
    \end{enumerate}
        
    \end{proposition}
    \begin{proposition}[{\cite[Proposition 2.8]{guedj2022quasi2}}]\label{Domination principle-1} Let $\theta$ be a real smooth semipositive $(1,1)$ form.
    Assume $\theta$ is non-collapsing.
 Suppose $u,v\in\operatorname{PSH}(X,\theta)\cap L^\infty(X)$, and 
\[
\mathbf 1_{\{u<v\}}\theta_u^n\le c\mathbf 1_{\{u<v\}}\theta_v^n
\]
for some $0\leq c<1$. Then
$u\ge v$. \end{proposition}
\begin{corollary}[{\cite[Corollary 2.9]{guedj2022quasi2}}]
\label{coro: domination-semipositive}
Let $\theta$ be a real smooth semipositive $(1,1)$ form.
    Assume $\theta$ is non-collapsing and
let $u,v\in\PSH(X,\theta)\cap L^\infty(X)$. \begin{enumerate}
    \item If $\theta_u^n\leq c\theta_v^n$, then $c\geq 1$.
    \item If $\theta_u^n\leq e^{\lambda(u-v)}\theta_v^n$ for some $\lambda>0$, then $u\geq v$. 
\end{enumerate}
    
\end{corollary}
\begin{definition}\label{def: nomalizing-constant} Let $\theta$ be a real smooth semipositive (1,1) form. Let $f$ be a positive measurable  function. Whenever there exists   $u\in\PSH(X,\theta)\cap L^\infty(X)$ such that
\[ (\theta+\dc u)^n=cfdV.\]
We denote by $c(\theta,f)$ the corresponding constant. 
\end{definition}
\begin{lemma}\label{lem: nomalizing-constant} The constant $c(\theta,f)$ is unique.
If $\int_Xf^{1/n}dV>0$ then 
     $c(\theta,f)\leq C$,  where $C$ depends on $X$, $A\omega_X$, $n$ and a positive lower bound for $\int_Xf^{1/n}dV$.
\end{lemma}
\begin{proof}
    The uniqueness follows from the domination principle; Corollary \ref{coro: domination-semipositive} (1). The domination principle again shows the monotonicity property:
    \[\theta\leq\theta'\Longrightarrow c(\theta,f)\leq c(\theta',f). \]
    If $\theta\leq A\omega_X$ for $A>0$ then $c(\theta,f)\leq c(A\omega_X,f)$. It follows from \cite[Lemma 5.9]{kolodziej2015weak} (see also~\cite[Corollary 3.3]{guedj2021quasi}) that $c(A\omega_X,f)\leq C$, where $C$ depends on $X$, $\omega_X$, $n$ and a positive lower bound for $\int_Xf^{1/n}dV$.
\end{proof}
\subsubsection{Big forms}
\begin{definition}\label{def: big}
		We say that $\theta$ is {\em big} if there exist
$\rho_0\in\PSH(X,\theta)$ with solvable singularities and
$\varepsilon_0>0$ such that $\theta+\dc \rho_0\geq \varepsilon_0\omega_X$ for some $\varepsilon_0>0$. We let $\Omega$ denote the open Zariski set in which $\rho_0$ is smooth.
	\end{definition}

    Thanks to Demailly--Pa\u un's regularization theorem (Theorem~\ref{thm: demailly-regularization}), since $\theta$ is big there always exists a positive current $T=\theta+\dc \Phi$ with solvable singularities. As shown in~\cite[Proposition 3.12]{BoucksomGuedjLu2025-volume}, we can pick a modification $\pi:Y\to X$ such that $\pi^*T$ has divisorial singularities and admits a
decomposition
    \begin{equation}\label{eq: decomposition-current}\pi^*T=\omega_Y+[E],
    \end{equation} where $\omega_Y$ is a Hermitian form on $Y$ and $E$ is an effective $\mathbb{R}$-divisor and $[E]$ is the current of integration along $E$.

\medskip
We observe that $V_\theta\geq \rho_0$, in particular $V_\theta$ is locally bounded in $\Omega$. 
 As in~\cite[Remark 2.10]{BoucksomGuedjLu2025-volume}, for any $u\in\PSH(X,\theta)\cap L^\infty_{\rm loc}(\Omega)$, the Monge-Amp\`ere measure $\theta_u^n=(\theta+\dc u)^n$ can be understood as in the classical sense of Bedford–Taylor in $\Omega$, and can thus be viewed as a
nonpluripolar measure on $X$.   

\medskip
The following is due to~\cite[Theorem 2.9]{BoucksomGuedjLu2025-volume}.
    \begin{theorem}\label{thm: envelope-big} Let $\theta$ be a big real smooth $(1,1)$ form.
		If $h$ is quasi-continuous,\ and  $P_\theta(h)\in\PSH(X,\theta)$, then $(\theta+\dc P_\theta(h))^n$ is supported on the contact set $\{P_\theta(h)=h\}$.
	\end{theorem}
{
\begin{lemma}\label{lem: eq-envelope} Let $\theta$ be a big real smooth $(1,1)$ form. Let $h\in\mathcal{C}^{\infty}(X,\mathbb{R})$ be a quasi-psh such that $P_\theta(h)\in\PSH(X,\theta)$. Then \[(\theta+\dc P_\theta(h))^n= \mathbf{1}_{\{P_\theta(h)=h \}}(\theta+\dc h)_+^n. \]
\end{lemma}
\begin{proof}
    When $\theta $ is Hermitian, the result follows from ~\cite{ChuZhou19regularity,ChuMcCleerey21-fully}. Set $\phi=P_\theta(h)$. Using ~\cite[Theorem 6.1]{ChuMcCleerey21-fully} and arguments in ~\cite[Theorem 2.3, Step 1]{darvas2025relative} we obtain
    \begin{align}\label{eq: envelope}
        \mathbf{1}_{\{ \phi=h\} }(\omega+\dc\phi)^n=\mathbf{1}_{\{ \phi=h\} }(\omega+\dc h)^n_+,
    \end{align} where $\omega$ is a Hermitian form such that $\omega>\theta$ and $h$ is $\omega$-psh.
     Let $A>0$ such that $\theta+A\omega_X$ is Hermitian and $h$ is $(\theta+A\omega_X)$-psh. 
    Let $\chi\in\mathcal{C}^0(X,\mathbb R)$. 
    We consider for $t\geq 0$,
    \[Q(t)=\int_{ \{\phi=h\} }\chi(\theta+t\omega_X+\dc P_\theta(h))^n- \int_{ \{\phi=h\}}\chi(\theta+t\omega_X+\dc h)^n.\]
Then $Q(t)$ is a polynomial
in $t$. By\eqref{eq: envelope}, we have for $t>A$
\[ \mathbf{1}_{\{ \phi =h\}}(\theta+t\omega_X+\dc \phi)^n=\mathbf{1}_{\{ \phi =h\}}(\theta+t\omega_X+\dc h)^n.\]
It follows that the polynomial $Q(t)$ is identically zero for $t>A$. Hence, $Q(t)\equiv 0$, so $Q(0)=0$. Since $\chi$ is taken arbitrarily, we
have the desired equality between measures. 
\end{proof}
}
    We recall the following maximum principle.
	\begin{lemma}[{\cite[Proposition 2.11]{BoucksomGuedjLu2025-volume}}]\label{lem: max-princ-big} Let $\theta$ be a big (1,1) form.
		Let $u,v$ be $\theta$-psh functions such that $\rho_0-C\leq u\leq v$. Then
		\begin{equation*}
		\mathbf{1}_{\{u=v\}}(\theta+\dc u)^n\leq\mathbf{1}_{\{u=v\}}(\theta+\dc v)^n.
		\end{equation*}
	\end{lemma}

     \begin{proposition}[{\cite[Corollary 2.13]{BoucksomGuedjLu2025-volume}}]\label{lem: BGL25-213} Let $\theta$ be a big (1,1) form. Fix $\lambda\geq 0$. Let $u,v\in\PSH(X,\theta)$ such that $\min (u,v)\geq \rho_0-O(1)$. 
    \begin{enumerate}
        \item If $\theta_u^n\geq e^{\lambda u}fdV$ and $\theta_v^n\geq e^{\lambda v}gdV$ then
        \[\theta_{\max(u,v)}^n\geq e^{\lambda\max(u,v)}\min(f,g)dV. \]
        \item If $\theta_u^n\leq e^{\lambda u}fdV$ and $\theta_v^n\leq e^{\lambda v}gdV$ then
        \[\theta_{P_\theta(\min(u,v))}^n\leq e^{\lambda P_\theta(\min(u,v))}\max(f,g)dV. \]
    \end{enumerate}
        
    \end{proposition}

    \begin{proposition}[{\cite[Lemma 4.2]{BoucksomGuedjLu2025-volume}}]\label{Domination principle-big} Let $\theta$ be a big (1,1) form.
 Suppose $u,v\in\operatorname{PSH}(X,\theta)$ such that $O(1)+\min(u,v)\geq \rho_0$ and $v\leq  u+O(1)$. If
\[
\mathbf 1_{\{u<v\}}\theta_u^n\le c\mathbf 1_{\{u<v\}}\theta_v^n
\]
for some $0\leq c<1$, then
$u\geq v$.
\end{proposition}
\begin{corollary}[{\cite[Corollaries 4.3, 4.4]{BoucksomGuedjLu2025-volume}}]
\label{coro: domination-big}Let $\theta$ be a big (1,1) form. Assume $u,v\in\PSH(X,\theta)$ are such that $u=v+ O(1)$, and $O(1)+\min(u,v)\geq \rho_0$. \begin{enumerate}
    \item If $\theta_u^n\leq c\theta_v^n$, then $c\geq 1$.
    \item If $\theta_u^n\leq e^{\lambda(u-v)}\theta_v^n$ for some $\lambda>0$, then $u\geq v$. 
\end{enumerate}
    
\end{corollary}

\section{A priori estimates: from local to global}
We first establish uniform estimates for solutions to the complex Monge--Amp\`ere equation in the local setting.
\begin{theorem}\label{thm: estimate-local}
    Assume $\mu=fdV$ is a positive measure in $\Omega$ with $\mu(\Omega)\leq 1$, where
 $f\in L^w(\Omega,dV)$ with $w$ satisfying \hyperlink{conditionK}{Condition (K)}. Then there exists a unique bounded function $\varphi\in\mathcal{F}^a(\Omega)$ such that $(\dc \varphi)^n=\mu$. More precisely, $\sup_\Omega|\varphi|\leq C$ for $C>0$ depending only on $\Omega$, $dV$, $n$, $w$ and an upper bound
for $\|f\|_w$.
\end{theorem}

\begin{lemma}\label{lem: finite-integral}
     Assume $\mu=fdV$ is a positive measure in $\Omega$ with $\mu(\Omega)\leq 1$, $f\in L^p(\Omega,dV)$ for $p>1$. Then
     \[A_m(\mu):=\sup\left\{\int_\Omega|u|^m d\mu: u\in\mathcal{E}_0(\Omega),\int_\Omega(\dc u)^n\leq 1 \right\}<+\infty,\;\forall\, m>0.\]
\end{lemma}
\begin{proof}
    See \cite[Lemma 3.2]{DL25-singularCY}.
\end{proof}
\begin{lemma}\label{lem: estimate}
    Assume that $f\in L^p(\Omega, dV)$ for some $p>1$, and that $u\in\PSH(\Omega)$ is normalized so that
$\sup_\Omega u=0$ and $u=0$ on $\partial\Omega$. Suppose that there exists a bounded function $v\in\PSH(\Omega)$
normalized so that $\sup_\Omega v=0$ and $v=0$ on $\partial\Omega$, such that
$$(\dc u)^n\leq (\dc v)^n+fdV$$
then $u\geq v-C$, where $C>0$ depends on $n$, $p$, $\|f\|_p$.
\end{lemma}
\begin{proof}
Let $\varphi\in\PSH(\Omega)\cap L^\infty(\Omega)$ be the solution to the Monge–Amp\`ere equation
\[(\dc \varphi)^n=fdV\;\;\text{in}\;\Omega\; \text{and}\; \varphi|_{\partial\Omega}=0. \]
We have a uniform bound for $\varphi$, which follows from Step 1 in Theorem~\ref{thm: estimate-local}. Consider now
$\psi=v+\varphi$. It is sufficient to show that $u\geq \psi$. Indeed, $\psi$ satisfies
$(\dc\psi)^n\geq (\dc v)^n+(\dc \varphi)^n=(\dc v)^n+fdV\geq (\dc u)^n.$
Thus, it follows from the Comparison Principle (Theorem~\ref{thm: local-comparison}) that $u\geq \psi$ in $\Omega$. 
\end{proof}

\begin{proof}[Proof of Theorem ~\ref{thm: estimate-local}]
    \textbf{Step 1:} $w(t)=t^p$ for some $p>1$. The existence of $\varphi\in\mathcal{F}(\Omega)$ satisfying $(\dc \varphi)^n=\mu$ follows from Theorem~\ref{thm: cegrell}. It suffices to prove that $\varphi$ is uniformly bounded.

    \textit{Step 1.1}  First, we assume that $\varphi\in\mathcal{E}_0(\Omega)$. We will show that $\sup_\Omega |\varphi|\leq C$ for a uniform constant $C>0$.
We set  $ \varphi_s=\varphi-\inf_\Omega\varphi-s$ for $0<s<-\inf_\Omega\varphi$ and $\Omega_s:=\{\varphi_s<0\}$. Fix $r>0$, we set $$A_s=\int_{\Omega_s}(-\varphi_s)^rd\mu.$$ If $A_s=0$, then $(\dc\varphi)^n=0$ on $\Omega_s$, and Theorem~\ref{thm: local-comparison}, applied to $\varphi$ and the constant $\inf_\Omega\varphi+s<0$, gives $\varphi\geq\inf_\Omega\varphi+s$, a contradiction. Otherwise, $A_s>0$. 

We set $\varepsilon=\left( \frac{n+r}{n}\right)^{\frac{n}{n+r}}A_s^{\frac{1}{n+r}}$. Let $\chi:\mathbb R \to\mathbb R^+$ be the  function defined by 
\begin{equation*}
    \chi(t)=\begin{cases}
        \varepsilon^{-\frac{n+r}{n}}t^{\frac{n+r}{n}},& t>0\\
        0,& t\leq 0,
    \end{cases}
\end{equation*}
which is an increasing convex function.  Then $\gamma(t)=\varepsilon t^{\frac{n}{n+r}}$ for $t>0$ and $\gamma(t)=0$ for $t\leq 0$, is the inverse function of $\chi$, which is an increasing concave function.

We set $\psi:=-\chi (-\varphi_s)$ and $u=P(\psi)$. Put $M_s=\chi'(s)$. Since $(-\varphi_s)_+\leq\min\{s,-\varphi\}$ and $\chi$ is convex, we have $M_s\varphi\leq\psi\leq0$. Thus $M_s\varphi\leq u\leq0$, so $u$ is a bounded psh function with zero boundary data.
Let $\mathcal C=\{u=\psi\}$ and $\mathcal C_+=\mathcal C\cap\Omega_s$. The measure $(\dc u)^n$ is concentrated on $\mathcal C$. Since $A_s>0$, the obstacle is negative somewhere, so $u\not\equiv0$ and the maximum principle gives $u<0$ on $\Omega$. Hence $\mathcal C\subset\Omega_s$, and $(\dc u)^n$ is concentrated on $\mathcal C_+$.
The function $v:=-\gamma(-u)$ satisfies
\[v\leq -\gamma(-\psi)=\min\{\varphi_s,0\}\leq \varphi_s \]
on $\Omega$. A direct computation shows that
\[\dc(-\gamma(-u))=\gamma'(-u)\dc u-\gamma''(-u)du\wedge d^c u \geq \gamma'(-u)\dc u,\]
hence \[(\dc v)^n\geq (\gamma'(-u))^n(\dc u)^n. \]
By the maximum principle (Lemma~\ref{lem: max-princ-local}), we have
\begin{align*}
    (\chi'(-\varphi_s))^{-n}(\dc u)^n&=(\chi'(-\varphi_s))^{-n}\mathbf{1}_{\mathcal{C}_+}(\dc u)^n\\
    &= (\gamma'(-u))^n \mathbf{1}_{\mathcal{C}_+}(\dc u)^n\\
    &\leq   \mathbf{1}_{\mathcal{C}_+}(\dc v)^n\leq\mathbf{1}_{\mathcal{C}_+}(\dc \varphi)^n.
\end{align*}
   Since $(\dc u)^n$ is concentrated on $\mathcal C_+$, it follows that
    $$(\dc u)^n\leq \mathbf{1}_{\mathcal C_+}(\chi'(-\varphi_s))^n(\dc\varphi)^n\leq A_s^{-1}\mathbf{1}_{\Omega_s}(-\varphi_s)^rfdV.$$
Since $u$ is a bounded psh function with zero boundary data by construction, by Lemma~\ref{lem: finite-integral}, for any $m>0$ there is $A_m>0$ depending on  $\Omega$, $dV$, $m$, $p$, and $\|f\|_p$ such that
\begin{equation}
    \int_{\Omega_s}(-u)^m fdV\leq A_m<+\infty.
\end{equation} 
By comparison, we have $u\leq \psi$, that is
\begin{equation}\label{eq: lem-key-32}
     \frac{n}{n+r}A_s^{-\frac{1}{n}}(-\varphi_s)^{\frac{n+r}{n}} \leq -u.
\end{equation}


Let $p^*=\frac{p}{p-1}$.
We apply the inequality above with $r=\frac{1}{p^*}$. If $-\inf_\Omega\varphi \leq1$, we are done. Otherwise, fix $s_0=1<-\inf_\Omega\varphi$ and work with $0<s\leq s_0$.
By the H\"older inequality, we have
\begin{equation*}
    A_s\leq \int_\Omega(-\varphi)_+^{\frac{1}{p^*}}fdV\leq \|f\|_{p}\|\varphi\|_{1}^{1/p^*}.
\end{equation*} 
We define $\phi(s):=\int_{\Omega_s}fdV$. Inequality~\eqref{eq: lem-key-32} and H\"older's inequality yield
\begin{align*}
    A_s&\leq \left(\frac{n+r}{n}\right)^{\frac{nr}{n+r}}\int_{\Omega_s} A_s^{\frac{r}{n+r}}(-u)^{\frac{n}{n+r}}fdV\\
    &\leq 2^{\frac{nr}{n+r}}A_s^{\frac{r}{n+r}}\|f\|_{p}^{1/q} \|u\|_{\frac{nrqp^*}{n+r}}^{\frac{nr}{n+r}}\phi(s)^{1-\frac{1}{q}}\\
     &\leq C(n,p,q,\|f\|_p)A_s^{\frac{r}{n+r}}\phi(s)^{1-\frac{1}{q}}.
\end{align*} Hence $A_s\leq C_0\phi(s)^{(1-\frac{1}{q})\frac{n+r}{n}}$.
On the other hand, we have
\[A_s\geq \int_{\Omega_s}(-\varphi_s)fdV\geq\int_{\Omega_{s-t}}(-\varphi_s)fdV \geq t \int_{\Omega_{s-t}}fdV.\] 
 We choose $q>\frac{n}{r}+1=np^*+1$ so that $\delta_0=(1-\frac{1}{q})\frac{n+r}{n}-1\in(0,\frac{1}{np^*})$.
We have for $0<t<s\leq s_0$,
\[t\phi(s-t)\leq C_0(\phi(s))^{1+\delta_0} \]
where $\delta_0\in(0,\frac{1}{np^*})$ and $C_0$ depend on $n$, $dV$, $p$ and $\|f\|_p$. The local comparison principle gives $\phi(s)>0$ for $0<s\leq s_0$: otherwise $\varphi\geq\inf_\Omega\varphi+s$, a contradiction.
By De Giorgi's lemma (see, e.g.,~\cite[Lemma 7]{Guo-Phong2024-estimates}), we have 
\[s_0\leq C_1(\phi(s_0))^{\delta_0} \] with $C_1=\frac{2C_0}{1-2^{-\delta_0}}.$ On the other hand, H\"older's inequality yields
\begin{align*}
    \phi(s_0)\leq \int_\Omega\frac{(-\varphi)^{1/p^*}}{(-\inf_\Omega\varphi-s_0)^{1/p^*}}fdV\leq \frac{\|f\|_p\|\varphi\|_{L^1(\Omega,dV)}^{1/p^*}}{(-\inf_\Omega\varphi-s_0)^{1/p^*}}.
\end{align*} Thus, it follows that
\[-\inf_\Omega\varphi\leq s_0+\frac{C_2\|\varphi\|_{L^1(\Omega,dV)}}{s_0^{p^*/\delta_0}} \]
where $C_2$ depends on $C_1$, $\delta_0$, $p$ and $\|f\|_p$. Since $s_0=1$ and $\int_\Omega(\dc\varphi)^n\leq1$, the uniform $L^1$ bound for this family in $\mathcal E_0(\Omega)$ (Lemma~\ref{lem: finite-integral}) gives the required estimate.
    
\medskip
    \textit{Step 1.2} In this step we show that $\varphi$ is bounded.

We solve $(\dc\varphi_j)^n=\min(f,j)dV$ with $\varphi_j\in\mathcal{E}_0(\Omega)$. Since $(\dc\varphi_j)^n\leq fdV$ our above
estimate shows that $\|\varphi_j\|_\infty\leq C$ for a uniform constant $C>0$. The comparison principle
also implies that $\varphi_j$ is decreasing and $\varphi\leq \varphi_j$, thus $u=\lim_j\varphi_j$ is bounded and $(\dc u)^n=\mu$. It then follows from \cite[Theorem 5.15]{cegrell2004general} that $u=\varphi$, completing the proof.

    \bigskip
\noindent   \textbf{Step 2:} $w$ satisfies the \hyperlink{conditionK}{condition (K)}. 
We write
\[
w(t)\asymp_{+\infty}t h(\log(1+t)),\qquad \int_1^\infty h(s)^{-1/n}\,ds<+\infty,
\]
where $h:[0,+\infty)\to(0,+\infty)$ is increasing. Choose $c_w>0$ and $T\geq1$ such that $w(t)\geq c_w t h(\log(1+t))$ for all $t\geq T$. Set
\[
L=2\max\{1,\|f\|_w\},\qquad F=L^{-1}f.
\]
By the definition of the Luxemburg norm, $\int_\Omega w(F)\,dV\leq1$.
By Theorem~\ref{thm: cegrell}, there exists $v\in\mathcal{F}(\Omega)$ such that
\begin{equation}
    (\dc v)^n=w(F)dV,\qquad \sup_\Omega v=0.
\end{equation}
Its total Monge--Amp\`ere mass is at most one. Hence, by~\cite[Theorem B]{ahag2009partial}, we can choose $p>1$, $\alpha>0$, and $C_\alpha>0$, depending only on the fixed domain and volume form, such that
\[
\|e^{-\alpha v}\|_{L^p(\Omega,dV)}\leq C_\alpha.
\]
Define
\[
\gamma(t)=\frac{(L/c_w)^{1/n}}{\alpha}\int_0^{\alpha t}h(s)^{-1/n}\,ds,\qquad t\geq0.
\]
Then $\gamma$ increases and is concave, $\gamma(0)=0$, and
\[
0\leq\gamma(t)\leq M:=\frac{(L/c_w)^{1/n}}{\alpha}\int_0^\infty h(s)^{-1/n}\,ds<+\infty.
\]
Set $\psi=-\gamma(-v)$, with value $-M$ on $\{v=-\infty\}$. Then $\psi\in\PSH(\Omega)\cap L^\infty(\Omega)$, $-M\leq\psi\leq0$, and $\sup_\Omega\psi=0$. Moreover,
\[
(\dc\psi)^n\geq\bigl(\gamma'(-v)\bigr)^n(\dc v)^n.
\]
To justify this inequality when $v$ is unbounded, put $v_j=\max\{v,-j\}$ and $\psi_j=-\gamma(-v_j)$. For bounded $v_j$, the chain inequality follows by smooth approximation of the concave function $\gamma$. Plurifine locality (\cite{bedford1987fine}) then gives
\[
(\dc\psi_j)^n\geq\mathbf{1}_{\{v>-j\}}\bigl(\gamma'(-v)\bigr)^n(\dc v)^n.
\]
Now $\psi_j\searrow\psi$ with $-M\leq\psi_j\leq0$, so the Bedford--Taylor convergence theorem (\cite[Theorem 3.18]{guedj2017degenerate}) applies on the left. The measures on the right increase to $\bigl(\gamma'(-v)\bigr)^n(\dc v)^n$, since $(\dc v)^n=w(F)dV$ does not charge the pluripolar set $\{v=-\infty\}$.

Set $E:={\{F\geq T\}\cap}\{\log(1+F)\geq -\alpha v\}$.
On $E$, the monotonicity of $h$ yields
\begin{align*}
(\dc\psi)^n
\geq {\frac{L}{c_w}}\frac{w(F)}{h(-\alpha v)}\,dV
\mathrel{{\geq}} L\frac{Fh(\log(1+F))}{h(-\alpha v)}\,dV
\geq LF\,dV=f\,dV.
\end{align*}
On $\Omega\setminus E$, we have $f=LF\leq {LT} e^{-\alpha v}$ almost everywhere. Combining the two estimates gives
\[
(\dc\varphi)^n\leq(\dc\psi)^n+GdV,\qquad
G:={LT} e^{-\alpha v},\qquad \|G\|_p\leq {LT} C_\alpha.
\]
Lemma~\ref{lem: estimate} therefore yields $\varphi\geq\psi-C\geq-M-C$. Since $\varphi\leq0$, this proves the desired uniform bound, with constants depending only on the fixed domain, $dV$, $n$, $w$, and an upper bound for $\|f\|_w$. This completes the proof. \end{proof}
From $L^\infty$ estimate in the local setting, we can establish the one in the context of Hermitian manifolds.

\begin{lemma}\label{lem: subsolution} Let $(X,\omega_X)$ be an $n$-dimensional compact Hermitian manifold.
    Assume that the weight function $w$ satisfies \hyperlink{conditionK}{Condition (K)}.
    Then there exists a uniform constant $m=m(X,\omega_X,n, dV,w)$ such that for all $0\leq f\in L^w(dV)$, we can find $u\in\PSH(X,\omega_X)$ with $-1\leq u\leq 0$ and
    \[(\omega_X+\dc u)^n\geq m\frac{f}{\|f\|_w}dV. \]
\end{lemma}
\begin{proof}
The proof is almost identical to that of~\cite[Lemma 2.1]{guedj2021quasi}, which goes back to~\cite[Theorem 0.1]{kolodziej2005-measures}. See also~\cite[Lemma 2.3]{Fang25-continuity-hessian}.\end{proof}

\begin{theorem} \label{thm: estimate-hermit}
Let $0\leq f\in L^w$, with the weight function $w$ satisfying \hyperlink{conditionK}{Condition (K)}.
If $\varphi\in\PSH(X,\omega_X)\cap L^\infty(X)$ satisfies
    \begin{equation}\label{}
        (\omega_X+\dc\varphi)^n\leq fdV,\quad\sup_X\varphi=0,
    \end{equation} then $\textrm{Osc}(\varphi)\leq C$ for a constant $C>0$ depending only on $(X,\omega_X)$, $dV$,
$w$, and an upper bound for $\|f\|_w$.
\end{theorem}
\begin{proof} 
{The proof is similar to that of~\cite[Theorem 2.4]{Fang25-continuity-hessian}.
Fix $K\geq\max\{1,\|f\|_w\}$. By Lemma~\ref{lem: weak-weight},
we can choose a Young weight $\widetilde w$ satisfying
\hyperlink{conditionK}{Condition (K)} such that
$\frac{\widetilde w(t)}{w(t)}\to0$ as $t\to+\infty$.
Let $m>0$ be the constant in Lemma~\ref{lem: subsolution}
for the weight $\widetilde w$.
The $\Delta_2$-condition gives $\int_Xw(f)dV\leq A_K$ for a uniform
constant $A_K\geq1$, and, for every fixed $\varepsilon>0$,
$\frac{\widetilde w(t/\varepsilon)}{w(t)}\to0$ as $t\to+\infty$.
Taking $\varepsilon=2^{-n-1}m$, choose $B\geq1$, depending only on
$w$, $K$, and the fixed geometric data, such that
$\widetilde w(t/\varepsilon)\leq A_K^{-1}w(t)$ for $t>B$.
Then
\[
 \int_X\widetilde w\left(
 \frac{\mathbf{1}_{\{f>B\}}f}{\varepsilon}\right)dV\leq1,
 \qquad
 \|\mathbf{1}_{\{f>B\}}f\|_{\widetilde w}\leq \frac{1}{2^{n+1}}m.
\]
By Lemma~\ref{lem: subsolution}, applied with $\widetilde w$, we can
find $u\in\PSH(X,\om_X)$ such that $-1\leq u\leq0$ and
\[
 (\om_X+\dc u)^n\geq
 m\frac{\mathbf{1}_{\{f>B\}}f}
 {\|\mathbf{1}_{\{f>B\}}f\|_{\widetilde w}}dV.
\]
If $\mathbf{1}_{\{f>B\}}f=0$ almost everywhere, take $u=0$ instead.
We infer that
\[(\om_X+\dc\varphi)^n\leq \frac{1}{2^{n+1}}(\om_X+\dc u)^n+BdV. \]
Since normalized \(\omega_X\)-psh functions have a uniform \(L^1\) bound,
we can choose $C>0$ so large that $$\|\mathbf{1}_{\{\varphi<-C\}}\|_{\widetilde w}\leq m\frac{1}{2^{n+1}B}.$$ By Lemma~\ref{lem: subsolution}, applied with $\widetilde w$, there exists a function $v\in\PSH(X,\omega_X)$ such that $-1\leq v\leq 0$ and \[(\omega_X+\dc v)^n\geq m\frac{\mathbf{1}_{\{\varphi<-C\}}}{\|\mathbf{1}_{\{\varphi<-C\}}\|_{\widetilde w}}dV\geq \mathbf{1}_{\{\varphi<-C\}}2^{n+1}BdV. \]
If $\{\varphi<-C\}$ has zero volume, take $v=0$; the last measure inequality still holds.
 Then on the set $\{\varphi<\frac{u+v}{2}-C\}\subset\{\varphi<-C\}$, we have
 \begin{eqnarray*}
     (\omega_X+\dc\varphi)^n&\leq& \frac{1}{2^{n+1}}(\om_X+\dc u)^n+ \frac{1}{2^{n+1}}(\om_X+\dc v)^n\\
     &\leq& \frac{1}{2}\left(\omega_X+\dc\frac{u+v}{2} \right)^n.
 \end{eqnarray*} Then by the domination principle (Proposition~\ref{Domination principle-1}), $\varphi\geq \frac{u+v}{2}-C\geq -1-C$.}
\end{proof}

\begin{lemma}\label{lem: satbility0} Let $0\leq F\in L^w$, with the weight function $w$ satisfying \hyperlink{conditionK}{Condition (K)}.
Let $\lambda\geq 0$.
Let $u,v\in\PSH(X,\omega_X)\cap L^\infty(X)$ such that $\sup_Xu=-1$, $v\leq 0$ and
\[ (\omega_X+\dc u)^n=e^{\lambda u}FdV.\]
Then \[ \sup_X(v-u)_+\leq C\|(v-u)_+F\|_{w}^{\frac{1}{n+1}}, \]
    where  $C>0$ depends only on $\omega_X$, $dV$, $n$, $w$, $\|F\|_w$ and $\sup_X|v|$.

    In particular, if $0\geq v_j\in\PSH(X,\omega_X)\cap L^\infty(X)$ is uniformly bounded and 
    converges to $u$ in $L^1(dV)$, then {
    \[\sup_X\{v_j-u\}_+\to 0.\] If, in addition, $v_j\geq u$, then $v_j$ converges uniformly to $u$.}
\end{lemma}
\begin{proof}
    The first statement follows almost verbatim from~\cite[Theorem 3.1]{Fang25-continuity-hessian}, in which we apply Lemma~\ref{lem: subsolution} instead. 
 {   To see this, put
\[
 D:=\big\|(v-u)_+F\big\|_{ w},
 \quad
 E_\varepsilon:=\{u<v-\varepsilon\},
 \quad
 G_\varepsilon:=\mathbf 1_{E_\varepsilon}F,\quad M=\sup_X|v|.
\]
By monotonicity and homogeneity of the Luxemburg norm,
\[
       \|G_\varepsilon\|_{ w}
       \leq \varepsilon^{-1}D.
\]
If $G_\varepsilon=0$, then $\mathbf1_{E_\varepsilon}\omega_u^n=0$.
Proposition~\ref{Domination principle-1}, applied to $u$ and
$v-\varepsilon$ with coefficient zero, gives $u\geq v-\varepsilon$.
Thus, $\sup_X(v-u)_+\leq\varepsilon$, which is stronger than the
estimate below. We may therefore assume $G_\varepsilon\ne0$.
Let $\psi_\varepsilon$ be a subsolution constructed in
Lemma~\ref{lem: subsolution} for $G_\varepsilon$, and set
\[
 t={ \left(
    \frac{2}{m}
    \|G_\varepsilon\|_{ w}
 \right)^{1/n}.}
\]
If $t\geq1$, the estimate $\sup_X(v-u)_+\leq Ct+\varepsilon$ follows directly
from the boundedness of $u$ and $v$. We may therefore assume that
$0<t<1$ and define
\[
       V_t:=(1-t)v+t\psi_\varepsilon-A_0t,
\]
where $A_0$ is chosen, depending only on $M$, so that $V_t\leq v$.
On the set $\{u<V_t-\varepsilon\}\subset E_\varepsilon$, we have
\begin{align*}
     (\omega_X+\dc{V_t})^n
 &\geq t^n(\omega_X+\dc{\psi_\varepsilon})^n\\
 &{\geq 2F\,dV}\geq 2(\omega_X+\dc u)^n.
\end{align*}
The domination principle therefore gives
$u\geq V_t-\varepsilon$. In either case,
\[
 \sup_X(v-u)_+
 \leq Ct+\varepsilon
 \leq C\varepsilon^{-1/n}D^{1/n}+\varepsilon.
\]
Choosing $\varepsilon:=D^{\frac{1}{n+1}}$ gives the desired estimate.

{It remains to show that $\|(v_j-u)_+F\|_w\to 0$ as $j\to+\infty$. Assume $\|v_j\|_{L^\infty}+\|u\|_{L^\infty}\leq M_0$. Fixing $\eta>0$, we have
\[
    0\leq w\left(\frac{(v_j-u)_+F}{\eta}\right)\leq w\left(\frac{M_0F}{\eta}\right)\in L^1(dV),
\]
since $F\in L^w$ and $w$ satisfies the $\Delta_2$-condition. As $v_j\to u$ in $L^1$, every subsequence admits a further subsequence converging to $u$ almost everywhere. The dominated convergence theorem applied to these subsequences therefore yields
\[
    \int_X w\left(\frac{(v_j-u)_+F}{\eta}\right)dV\longrightarrow 0.
\]
Thus $\|(v_j-u)_+F\|_w\leq\eta$ for all sufficiently large $j$. Since $\eta>0$ was arbitrary, we conclude that $\|(v_j-u)_+F\|_w\to 0$ as $j\to+\infty$.}
}
    \end{proof}

\begin{theorem}\label{thm: exist-hermit} 
Let $\lambda\geq 0$.
Let $0\leq  f\in L^w(dV)$ and $\int_Xf\,dV>0$, with $w$ satisfying \hyperlink{conditionK}{Condition (K)}.
 Then
    \begin{enumerate}
        \item for any $\lambda>0$ there exists a unique $\varphi_\lambda\in\PSH(X,\om_X)\cap L^\infty(X)$ such that 
        \begin{equation}
            (\omega_X+\dc\varphi_\lambda)^n=e^{\lambda\varphi_\lambda}fdV.
        \end{equation}
        \item there exist a unique constant $c=c(\om_X,f)>0$ and a $\varphi\in\PSH(X,\om_X)\cap L^\infty(X)$ with $\sup_X \varphi=-1$ such that
        \begin{equation}
            (\om_X+\dc\varphi)^n=cfdV.
        \end{equation}
    \end{enumerate} There exist constants $C_0,C_1>0$, depending only on $X$, $\omega_X$, $dV$, $n$, $w$, and upper bounds for $\|f\|_{1/n}^{-1}$ and $\|f\|_w$, but independent of $\lambda$, such that the solution in (2) and the solutions in (1) satisfy, respectively,
    \[
      \|\varphi\|_{L^\infty(X)}\leq C_0,
      \qquad
      \|\varphi_\lambda\|_{L^\infty(X)}
      \leq C_0+\frac{C_1}{\lambda}\quad(\lambda>0).
    \]
    Moreover, such bounded solutions  are continuous.
\end{theorem}  The result is first proved in~\cite{kolodziej2021continuous}. Inspired by \cite[Theorem C]{Fang25-continuity-hessian}, we present the proof which is based on the domination principle and the construction of subsolutions (Lemma~\ref{lem: subsolution}), avoiding the use of the theory of Monge--Amp\`ere capacities. This approach also applies to the case of uniformly non-collapsing forms, producing
new (relative) a priori bounds, as well as existence results; cf. Theorem~\ref{thm: exist-noncollapsing}.
\begin{proof}
   


    \textbf{Step 1: $f$ is bounded.} Assume $f\leq M-1$ for some $M>1$. 
    Let $f_j$ be a sequence of smooth functions such that $f_j\to f$ in $L^2(dV)$ and $f_j\leq M$. 

(1) It follows from~\cite{tosatti2010complex} that 
there exists $\varphi_j\in\PSH(X,\omega_X)\cap \mathcal{C}^\infty(X)$ such that
    \[(\omega_X+\dc\varphi_{\lambda,j})^n=e^{\lambda\varphi_{\lambda,j}}f_jdV. \]
    The bounded subsolution $u$ obtained in Lemma~\ref{lem: subsolution} satisfies
    \[(\omega_X+\dc u)^n\geq mdV\geq e^{\lambda(u+\lambda^{-1}\ln (mM^{-1}))}f_jdV. \]
    It follows that 
    \[(\omega_X+\dc \varphi_{\lambda,j})^n\leq e^{\lambda(\varphi_{\lambda,j}-u-\lambda^{-1}\ln (mM^{-1}))}(\omega_X+\dc u)^n, \]
    hence $\varphi_{\lambda,j}\geq u+\lambda^{-1}\ln(mM^{-1})\geq -1+ \lambda^{-1}\ln(mM^{-1})$ for all $j\in\mathbb N$. To obtain an upper bound, by~\cite{tosatti2010complex} choose smooth $\rho_j\in\PSH(X,\omega_X)$ and $b_j>0$ such that
    \[
    (\omega_X+\dc\rho_j)^n=b_jf_jdV,\qquad \sup_X\rho_j=0.
    \]
    Lemma~\ref{lem: nomalizing-constant} gives $b_j\leq B$, uniformly in $j$, since $\|f_j\|_{1/n}$ is bounded away from zero. Theorem~\ref{thm: estimate-hermit}, applied with the weight $t^2$, then gives $\rho_j\geq-A$, uniformly in $j$. The domination principle yields
    \[
    \varphi_{\lambda,j}\leq\rho_j-\inf_X\rho_j+\lambda^{-1}\log b_j
    \leq A+\lambda^{-1}\log B.
    \]
    By Hartogs' lemma, up to extracting a subsequence, $\varphi_{\lambda,j}$ converges in $L^1$ and almost everywhere to some $\varphi_\lambda\in\PSH(X,\omega_X)\cap L^\infty(X)$, and $a_{\lambda,j}:=\sup_X\varphi_{\lambda,j}$ converges to $a_\lambda:=\sup_X\varphi_\lambda$. Let
    $u_j:=\varphi_{\lambda,j}-a_{\lambda,j}$.  Then $\sup_Xu_j=0$ and 
    \[(\omega_X+\dc u_j)^n=e^{\lambda\varphi_{\lambda,j}}f_jdV,\] with $e^{\lambda\varphi_{\lambda,j}}f_j$ uniformly bounded for fixed $\lambda>0$. Applying Lemma~\ref{lem: satbility0} with $\lambda=0$ and the weight $t^2$ in both directions gives
    \[
    \|u_j-u_k\|_\infty\leq C\left[
    \|(u_j-u_k)_+e^{\lambda\varphi_{\lambda,k}}f_k\|_2^{\frac1{n+1}}
    +\|(u_k-u_j)_+e^{\lambda\varphi_{\lambda,j}}f_j\|_2^{\frac1{n+1}}\right]
    \xrightarrow{j,k\to+\infty}0,
    \]
    by dominated convergence. Since $a_{\lambda,j}\to a_\lambda$,
    $\varphi_{\lambda,j}$ converges uniformly to $\varphi_\lambda$ and 
    \[(\omega_X+\dc\varphi_\lambda)^n=e^{\lambda\varphi_\lambda}fdV. \] The uniqueness follows immediately.
 
Once (2) is established below, the stated estimates follow for both bounded and general $f$ as follows. Let $\varphi$ be a solution of (2), normalized by $\sup_X\varphi=0$, and set
 \[
 \psi_-:=\varphi+\lambda^{-1}\log c,\qquad
 \psi_+:=\varphi-\inf_X\varphi+\lambda^{-1}\log c.
 \]
 Then
 \[
 e^{\lambda\psi_-}f\leq cf\leq e^{\lambda\psi_+}f,
 \qquad (\omega_X+\dc\psi_\pm)^n=cf\,dV.
 \]
 The domination principle gives $\psi_-\leq\varphi_\lambda\leq\psi_+$, hence
 \[
 \|\varphi_\lambda\|_\infty\leq\|\varphi\|_\infty+\lambda^{-1}|\log c|.
 \]
 Lemma~\ref{lem: subsolution} and Corollary~\ref{coro: domination-semipositive} give $c\geq m/\|f\|_w$, while Lemma~\ref{lem: nomalizing-constant} gives the upper bound for $c$. Theorem~\ref{thm: estimate-hermit}, applied to $cf$, therefore yields constants $C_0,C_1>0$ with the stated dependencies such that $\|\varphi\|_\infty\leq C_0$ and $|\log c|\leq C_1$, independently of $\lambda$.

\medskip \noindent (2) By \cite{tosatti2010complex}, we can find $c_j>0$ and $\varphi_j\in\PSH(X,\omega_X)\cap\mathcal{C}^\infty$ such that 
\[(\omega_X+\dc\varphi_j)^n=c_jf_jdV,\quad\sup_X\varphi_j=-1. \]
By Lemma \ref{lem: subsolution} and the domination principle (Corollary~\ref{coro: domination-semipositive}), $c_j\geq m$ for all $j$. 
Since $\|f_j\|_{1/n}$ is uniformly bounded from below, it follows from Lemma~\ref{lem: nomalizing-constant} that $c_j$ is uniformly bounded from above. Hence $c_j\to c$.

It follows from Theorem~\ref{thm: estimate-hermit} that $\varphi_j\geq -C$ for $C> 0$ depending on $X$, $dV$, $n$, $w$, and $\|f\|_w$, for all $j$. Up to extracting  a subsequence, Hartogs' lemma  yields $\varphi_j$ converges in $L^1(dV)$ and a.e. to some $\varphi\in\PSH(X,\omega_X)\cap L^\infty(X)$.
 Lemma~\ref{lem: satbility0} shows that \[\sup_X|\varphi_j-\varphi_k|\leq C[\|(\varphi_j-\varphi_k)_+f_j\|_2^{\frac{1}{n+1}}+  \|(\varphi_j-\varphi_k)_+f_k\|_2^{\frac{1}{n+1}}]\xrightarrow{j,k\to+\infty}0.\]Therefore,
$\{\varphi_j\}_{j\in\mathbb{N}}$ is a Cauchy sequence in $\PSH(X,\omega_X)\cap \mathcal{C}^0(X)$, hence $\varphi_j$ converges uniformly to $\varphi\in \PSH(X,\omega_X)\cap \mathcal{C}^0(X)$. 
By the continuity of the
complex Monge–Amp\`ere operator, we have  
\[(\omega_X+\dc\varphi)^n=cfdV. \]
The uniqueness of $c$ follows from the domination principle (Corollary~\ref{coro: domination-semipositive}).

\bigskip
\textbf{Step 2: general case.} Set $f_j:=\min(f,j)$.

(1) By Step 1, we can find $\varphi_j\in\PSH(X,\omega_X)\cap \mathcal{C}^0(X)$ such that \[(\omega_X+\dc\varphi_{\lambda,j})^n=e^{\lambda\varphi_{\lambda,j}}f_jdV. \]
If $k\geq j$, then $f_k\geq f_j$ and
\[
 \omega_{\varphi_{\lambda,k}}^n
 \geq
 e^{\lambda(\varphi_{\lambda,k}
              -\varphi_{\lambda,j})}
 \omega_{\varphi_{\lambda,j}}^n.
\]
The domination principle (Corollary~\ref{coro: domination-semipositive}) gives
$\varphi_{\lambda,k}\leq\varphi_{\lambda,j}$. Thus
$\varphi_{\lambda,j}$ decreases to a bounded
$\omega_X$-psh function $\varphi_\lambda$. 
By Lemma~\ref{lem: subsolution}, there exists $-1\leq v\leq 0$, such that
\[
(\omega_X+dd^cv)^n\geq \frac{mf}{\|f\|_w}dV\geq e^{\lambda(v
+\lambda^{-1}\ln \frac{m}{\|f\|_w})}f_jdV.
\]
Since $$(\omega_X+\dc\varphi_{\lambda,j})^n\leq e^{\lambda\sup_X\varphi_1}fdV,$$
we have \[\varphi_{\lambda,j}\geq v+\lambda^{-1}\ln \frac{m}{\|f\|_w},\]
and hence by Theorem~\ref{thm: estimate-hermit}, $\|\varphi_{\lambda,j}\|_\infty\leq C$ independent of $j$.
By the continuity of the
complex Monge–Amp\`ere along monotonic sequences, we have  
\[(\omega_X+\dc\varphi)^n=e^{\lambda\varphi}fdV. \]
The uniqueness follows immediately from Corollary~\ref{coro: domination-semipositive}. 

Moreover, set $a_j:=\sup_X\varphi_{\lambda,j}$, $a:=\sup_X\varphi_\lambda$,
$u:=\varphi_\lambda-a$, and $v_j:=\varphi_{\lambda,j}-a_j$.
Since $a_j\to a$, $v_j\to u$ in $L^1(dV)$.
The functions $v_j$ are uniformly bounded, $\sup_Xu=0$, $v_j\leq0$, and
\[
 (\omega_X+\dc u)^n=e^{\lambda u}F\,dV,\qquad F:=e^{\lambda a}f.
\]
Lemma~\ref{lem: satbility0} therefore gives $\sup_X(v_j-u)_+\to0$. Hence
\[
 0\leq\sup_X(\varphi_{\lambda,j}-\varphi_\lambda)
 \leq a_j-a+\sup_X(v_j-u)_+
 \longrightarrow0.
\]
Hence $\varphi_{\lambda,j}\to\varphi_\lambda$ uniformly, and
$\varphi_\lambda$ is continuous.

(2)   By Step 1, we can find $c_j>0$ and  $\varphi_j\in\PSH(X,\omega_X)\cap L^\infty(X)$ such that \[(\omega_X+\dc\varphi_j)^n=c_jf_jdV,\quad\sup_X\varphi_j=-1.\]
As above, $c_j$ is uniformly bounded and away from zero; passing to a subsequence,
we may assume $c_j\to c$. Since
\[(\omega_X+\dc\varphi_j)^n\leq MfdV, \] Theorem \ref{thm: estimate-hermit} gives $\varphi_j\geq -C$ for $C>0$ depending on $X$, $dV$, $n$, $w$, and $\|f\|_w$, independent of $j$. 
Up to extracting  a subsequence, Hartogs' lemma  yields that $\varphi_j$ converges in $L^1(dV)$ and a.e. to some $\varphi\in\PSH(X,\omega_X)\cap L^\infty$. 

Applying Lemma~\ref{lem: satbility0} with $\lambda=0$ to $\varphi_j+1$ and $\varphi_k+1$, and using $c_jf_j\leq Mf$, we obtain
\[
\|\varphi_j-\varphi_k\|_\infty\leq C\left(
\|(\varphi_j-\varphi_k)_+f\|_w^{\frac{1}{n+1}}
+\|(\varphi_k-\varphi_j)_+f\|_w^{\frac{1}{n+1}}\right),
\]
where $C$ is independent of $j,k$. Choose $A>0$ such that $|\varphi_j|,|\varphi|\leq A$. For every $\eta>0$, the $\Delta_2$ condition gives
\[
0\leq w\!\left(\frac{|\varphi_j-\varphi|f}{\eta}\right)
\leq w\!\left(\frac{2Af}{\eta}\right)\in L^1(dV).
\]
Dominated convergence theorem and the definition of the Luxemburg norm imply $\||\varphi_j-\varphi|f\|_w\to0$. The triangle inequality therefore shows that the right-hand side of the stability estimate tends to zero as $j,k\to\infty$. Thus $\varphi_j$ converges uniformly to $\varphi\in\mathcal{C}^0(X)$. Since $c_jf_j\to cf$ in $L^1(dV)$, continuity of the complex Monge--Amp\`ere operator under uniform convergence yields
\[ (\omega_X+\dc\varphi)^n=cfdV.\]
Finally, $\sup_X\varphi=-1$, so replacing $\varphi$ by $\varphi+1$ gives the required normalization.
Moreover, let $v$ be any bounded solution of the equation 
\[
                     (\omega_X+\dc v)^n=cf\,dV.
\]
Then
\(
                     \widetilde f:=ce^{-v}f
\)
belongs to $L^w(dV)$. By the already proved case $\lambda=1$, there
is a continuous solution $u$ of
\[
             (\omega_X+\dc u)^n=e^u\widetilde f\,dV
             =e^{u-v}(\omega_X+\dc v)^n.
\]
The domination principle, applied in both directions, gives
$u=v$. Hence every bounded solution of the $\lambda=0$ equation is
continuous as well.
\end{proof}

 \section{Semipositive and uniformly non-collapsing forms}
Assume in this section that $\theta$ is semipositive. A natural idea is to use Hermitian form to approximate $\theta$.  We first need to establish the uniform estimate with $\|f\|_{_{\frac{1}{n}}}^{-1}$ replaced by $v_-(\theta)$.
 


\begin{theorem}\label{thm: estimate-noncollapsing}
    Assume that $\mu= fdV$ is a probability measure on $X$, and $0\leq f\in L^w(dV)$, where $w$ satisfies \hyperlink{conditionK}{Condition (K)}. Assume $\theta$ is a {\em semipositive} $(1,1)$ form on $X$ such that $v_-(\theta)>0$. Let $\omega$ be a Hermitian form such that $\theta\leq \omega\leq A\omega_X$ for some $A>0$.
    Then any solution $\varphi\in \PSH(X,\omega)\cap L^\infty(X)$ to  
    \[(\omega+\dc\varphi)^n=c\mu,\quad\sup_X\varphi=0 \] where $c>0$, satisfies $Osc(\varphi)\leq C$, for some uniform constant $C$ 
   depending only on $(X,\omega_X)$, $dV$, $n$, $A$,
$v_-(\theta)^{-1}$, $w$, and an upper bound for $\|f\|_w$.
\end{theorem}

The proof follows the method from~\cite{DangZhangZhou26-estimate-big}.
We set the sub-level set $\Omega_s:
=\{\varphi<-s\}$ for $s>0$. 
Fix $r>0$ and set $A_s=\int_{\Omega_s}(-\varphi-s)^rfdV$. If $A_s=0$, then $\omega_\varphi^n=0$ on $\Omega_s$, and the domination principle (Proposition~\ref{Domination principle-1}), applied to $\varphi$ and the constant $-s$, gives $\varphi\geq-s$. We may therefore restrict the auxiliary construction to $A_s>0$.
We solve an auxiliary complex Monge--Amp\`ere equation 
\begin{equation}
    (\omega+\dc v_s)^n=c_s\frac{\mathbf{1}_{\Omega_s}(-\varphi-s)^r }{A_s}fdV,\; v_s\in\PSH(X,\omega)\cap L^\infty(X),\quad\sup_X v_s=0,
\end{equation} where $A_s=\int_{\Omega_s}(-\varphi-s)^rfdV$, $c_s>0$. The existence of solution $v_s$ follows from Theorem~\ref{thm: exist-hermit}. By \cite[Proposition 3.2]{guedj2022quasi2}, $c_s\geq v_-(\omega)\geq v_-(\theta)>0$.
We compare the auxiliary function with the solution.
\begin{lemma}\label{lem: key}
With the assumptions in Theorem~\ref{thm: estimate-noncollapsing} and $A_s>0$, we have on $\Omega_s$
    \begin{equation}\label{eq: key-equation}
       \frac{n}{n+r}\left(\frac{2cA_s}{c_s}\right)^{-\frac{1}{n}}
(-\varphi-s)^{\frac{n+r}{n}}\le -v_s+\frac{n}{n+r}\left(\frac{2cA_s}{c_s}\right)^{\frac{1}{r}}.
    \end{equation}
\end{lemma}

\begin{proof}
    Set $$\varepsilon
=\left(\frac{n+r}{n}\right)^{\frac{n}{n+r}}
\left(\frac{2cA_s}{c_s}\right)^{\frac{1}{n+r}},\ \ 
\Lambda=\frac{n}{n+r}\left(\frac{2cA_s}{c_s}\right)^{\frac{1}{r}}$$ so that $\frac{\varepsilon n}{n+r}\Lambda^{-\frac{r}{n+r}}=1$.  
    Set $C_\Lambda:=\varepsilon\Lambda^{\frac{n}{n+r}}=\left(\frac{2cA_s}{c_s}\right)^{\frac{1}{r}}$. We introduce the continuous truncation 
\[ \chi(t) := 
\begin{cases} 0, & t\leq C_\Lambda,\\[1mm]\varepsilon^{-\frac{n+r}{n}}t^{\frac{n+r}{n}}-\Lambda , & t>C_\Lambda. 
\end{cases} \] 
The function $\chi$ is continuous, convex, and increasing. 
Let $\gamma:[0,+\infty)\to[C_\Lambda,+\infty)$ be the inverse of $\chi|_{[C_\Lambda,+\infty)}$, which is concave and increasing. By computations, we have
    \[ \gamma(t)=\varepsilon(t+\Lambda)^{\frac{n}{n+r}}\] $\gamma(0)=C_\Lambda$, $\gamma'(0)=1$ and $0<\gamma'\leq 1$.

Set $\widetilde\Omega_s:=\{C_\Lambda<-\varphi-s\}$.
We see that $\widetilde\Omega_s\subset\Omega_s$. We set $$\psi:= -\chi(-\varphi-s) \; \text{and}\; u:=P_\omega(\psi).$$ 
By Theorem~\ref{thm: exist-hermit}, $\varphi$, and hence $\psi$, is continuous.
Thus $u\leq\psi$ everywhere on $X$ and $u$ is bounded.
We denote by $\mathcal{C}=\{u=\psi\}$ the contact set. By Theorem~\ref{thm: envelope}, we have that $\omega_u^n$ is concentrated on the contact set $\mathcal{C}$. 

Set $v:=-\gamma(-u)$.
The identity $\gamma(\chi(t))=\max\{t,C_\Lambda\}$ and $u\leq\psi$
give
\[
 v\leq-\gamma(\chi(-\varphi-s))
 =\min\{\varphi+s,-C_\Lambda\}\leq\varphi+s,
\]
with equality on $\mathcal C_+:=\mathcal C\cap\widetilde\Omega_s$.
Next, we show that 
\begin{equation}\label{eq: active-contact} 
\{u<v_s\}\cap\mathcal C \subset\mathcal C_+. 
\end{equation} 
Indeed, let $x\in\mathcal C\setminus\mathcal C_+$. 
Then $-\varphi-s\leq C_\Lambda$, so $\chi(-\varphi-s)=0$ and hence $u(x)=\psi(x)=0$.  
Since $v_s\in\PSH(X,\omega)$ and $\sup_Xv_s=0$, 
we have $v_s\leq0$. 
Thus $u(x)\geq v_s(x)$, proving \eqref{eq: active-contact}. By computations, we have
\[ \omega_v^n \geq \bigl(\gamma'(-u)\bigr)^n\omega_u^n. \] 
 By \eqref{eq: active-contact}, the restriction of $\omega_u^n$ to $\{u<v_s\}$ is concentrated on $\mathcal C_+$. On $\mathcal C_+$ we have $\chi'(-\varphi-s)>0$, and, with the inverse derivative evaluated only on this set,
    \begin{equation}
        \begin{split}
        \mathbf 1_{\mathcal C_+} (\chi'(-\varphi-s))^{-n}\omega_u^n &= \mathbf 1_{\mathcal C_+} \bigl(\gamma'(-u)\bigr)^n\omega_u^n\\ &\leq \mathbf 1_{\mathcal C_+}\omega_v^n\\
        &\leq 1_{\mathcal C_+}\omega_\varphi^n,
        \end{split}
    \end{equation} where in the last inequality we used Lemma~\ref{lem: max-princ} since $v=\varphi+s$ on $\mathcal{C}_+$. Multiplying on $\mathcal C_+$ by $(\chi'(-\varphi-s))^n$ and using the preceding support property, we obtain on $\{u<v_s\}$,
    \begin{align*}
        \omega_u^n&\leq c_s(2cA_s)^{-1}\mathbf{1}_{\widetilde\Omega_s}(-\varphi-s)^{r}\omega_\varphi^n\\
        &\leq c_s(2cA_s)^{-1}\mathbf{1}_{\Omega_s}(-\varphi-s)^{r}cfdV \\
        &= \frac{1}{2}\omega_{v_s}^n.
    \end{align*}
 We apply the domination principle (Proposition~\ref{Domination principle-1}) to obtain
    $v_s\leq u $. Since $u\le\psi$, we have
$v_s\le\psi$. On $\widetilde\Omega_s$, this gives
\[
v_s
\le-\frac{n}{n+r}\left(\frac{2cA_s}{c_s}\right)^{-\frac{1}{n}}
(-\varphi-s)^{\frac{n+r}{n}}
+\frac{n}{n+r}\left(\frac{2cA_s}{c_s}\right)^{\frac{1}{r}}.
\]
On $\Omega_s\setminus\widetilde\Omega_s$, we have
\[
0<-\varphi-s\le\left(\frac{2cA_s}{c_s}\right)^{\frac{1}{r}},
\]
and therefore
\[
\frac{n}{n+r}\left(\frac{2cA_s}{c_s}\right)^{-\frac{1}{n}}
(-\varphi-s)^{\frac{n+r}{n}}\le\frac{n}{n+r}\left(\frac{2cA_s}{c_s}\right)^{\frac{1}{r}}.
\]
The desired inequality follows.
\end{proof}

\begin{proof}[Proof of Theorem~\ref{thm: estimate-noncollapsing}]
Before separating the two cases, we record a uniform upper bound for
the normalizing constant $c$.  Set
\[
 L:=2\max\{1,\|f\|_w\},\qquad F:=L^{-1}f.
\]
By the definition of the Luxemburg norm, $\int_Xw(F)\,dV\leq1$.
Since $w(t)/t$ is increasing and tends to $+\infty$, we can choose
$R>1$, depending only on $L$ and $w$, such that
$LR/w(R)\leq1/2$.  It follows that
\[
 \int_{\{F>R\}}f\,dV
 \leq L\frac{R}{w(R)}\int_Xw(F)\,dV\leq\frac12.
\]
Since $f\leq LR$ on $\{F\leq R\}$ and $\int_Xf\,dV=1$, we have
\[
 \int_Xf^{1/n}\,dV
 \geq\frac12(LR)^{-(n-1)/n}.
\] Lemma~\ref{lem: nomalizing-constant} yields that $c$ is uniformly bounded from above. While, Lemma~\ref{lem: monotonicity-volume} shows that $c\geq v_-(\theta)>0$.

  \textbf{Step 1:}  We first treat the case that $w(t)=t^p$; i.e., $f\in L^p(dV)$ for some $p>1$.

  Assume that $\theta\leq A\omega_X$ for $A>1$. Let $p^*$ be the conjugate exponent of $p$, i.e., $\frac{1}{p}+\frac{1}{p^*}=1$. We apply Lemma~\ref{lem: key} with $r=\frac{1}{p^*}$.
By the H\"older inequality, we have
\begin{equation*}
    A_s\leq \int_X(-\varphi)^{{1}/{p^*}}fdV\leq \|f\|_{p}\cdot\|\varphi\|_{1}^{1/p^*}.
\end{equation*} 
Since $\varphi$ belongs to the compact family of
$A\omega_X$-psh functions normalized by $\sup_X u=0$, its follows from ~\cite[Theorem 8.19]{dinew2019lectures} that their
$L^1(dV)$ norms are uniformly bounded, 
by a positive constant only depending on $dV$, $\theta$, $A,\omega_X$. Hence, since $c_s\geq v_-(\theta)$, we obtain $$\Lambda{=\frac{n}{n+r}(2c c_s^{-1}A_s)^{\frac{1}{r}}}\leq C\|f\|_p^{p^*
}.$$ 

Define $\phi(s):=\int_{\Omega_s}fdV$. Fix $q>1$ to be chosen later. 
 Lemma~\ref{lem: key} (with $r=\frac{1}{p^*}$) and H\"older's inequality yield
\begin{align*}
    A_s&\leq C_{n,r}\left(\frac{2c}{v_-(\theta)}\right)^{\frac{r}{n+r}}A_s^{\frac{r}{n+r}}\int_{\Omega_s}(-v_s+\Lambda)^{\frac{nr}{n+r}}fdV\\
    &\leq C_{n,r}\left(\frac{2c}{v_-(\theta)}\right)^{\frac{r}{n+r}}A_s^{\frac{r}{n+r}}\left(\int_{\Omega_s}(-v_s+\Lambda)^{\frac{nrq}{n+r}}fdV \right)^{1/q}\left(\int_{\Omega_s}fdV \right)^{1-1/q}\\
    &\leq C_{n,r}\left(\frac{2c}{v_-(\theta)}\right)^{\frac{r}{n+r}}A_s^{\frac{r}{n+r}}\|f\|_{p}^{1/q}\|(-v_s+\Lambda)\|_{\frac{nrqp^*}{n+r}}^{\frac{nr}{n+r}}\phi(s)^{1-\frac{1}{q}}\\
    &\leq C_{n,r}\left(\frac{2c}{v_-(\theta)}\right)^{\frac{r}{n+r}}A_s^{\frac{r}{n+r}}\|f\|_{p}^{1/q} (\|v_s\|_{\frac{nrqp^*}{n+r}}+\Lambda\|1\|_{\frac{nrqp^*}{n+r}})^{\frac{nr}{n+r}}\phi(s)^{1-\frac{1}{q}}\\
     &\leq C(c,v_-(\theta),n,p,q,\|f\|_p)A_s^{\frac{r}{n+r}}\phi(s)^{1-\frac{1}{q}},
\end{align*} where $C_{n,r}=\left( \frac{n+r}{n}\right)^{\frac{nr}{n+r}}$. 
Hence $A_s\leq C_0\phi(s)^{(1-\frac{1}{q})\frac{n+r}{n}}$.
On the other hand, for any $t\in [0,1]$, we have
\[A_s =\int_{\Omega_s}(-\varphi-s)_+^rfdV\geq t^r\phi(s+t)\geq t\phi(s+t).\] It follows that 
\[t\phi(s+t)\leq C_0\phi(s)^{(1-\frac{1}{q})\frac{n+r}{n}},\]
where we choose $q>\frac{n}{r}+1=np^*+1$ so that $\delta_0=(1-\frac{1}{q})\frac{n+r}{n}-1\in(0,\frac{1}{np^*})$. 
We observe that  
\[ \phi(s)\leq \frac{1}{s^{1/p^*}}\int_{\Omega_s}(-\varphi)^{\frac{1}{p^*}}fdV\leq \frac{C_1}{s^{1/p^*}}\xrightarrow{s\to\infty}0,\] so there exists $s_0>0$ such that $\phi(s_0)^{\delta_0}<\frac{1}{2C_0}$.
Since $\phi$ is non-increasing and right-continuous we can apply De Giorgi's iteration lemma (see, e.g., ~\cite[Lemma 2.4]{eyssidieux2009singular}) to obtain that $\phi(s)=0$ for all $s\geq S_\infty=s_0+\frac{2C_0\phi(s_0)^{\delta_0}}{1-2^{-\delta_0}}$. 

Therefore, $\omega_\varphi^n(\varphi<-S_\infty)=0$,
by the domination principle, $\varphi\geq -S_\infty$ on $X$.

\medskip \textbf{Step 2:} we treat the general case $w(t)\asymp_{+\infty}th(\log t)$, where $h$ satisfies the integrability condition $\int_1^{+\infty}h^{-1/n}(t)dt<+\infty$. 

\medskip
\textit{Step 2.1. Reduction to the $L^p$ density.} 
 By Theorem \cite[Corollary 5.9]{ALEHYANELuSalouf26-Hermitian}, there exists $v\in\mathcal{E}(X,\omega)$ and a unique $c_1>0$ such that $\int_X\omega_v^n\geq v_-(\omega)$ and
\begin{equation}
    (\omega+\dc v)^n=c_1 cw\circ f dV,\quad\sup_Xv=0,
\end{equation}  where $\mathcal{E}(X,\omega)=\{u\in\PSH(X,\omega): P_\omega(bu)\in\PSH(X,\omega),\,\forall\, b\geq 1\}$; cf~\cite[page 17]{ALEHYANELuSalouf26-Hermitian}.
By Lemma \ref{lem: monotonicity-volume}, $v_-(\omega)\geq v_-(\theta)$ it follows  that $$c_1\geq \frac{v_-(\theta)}{c\int_Xw(f)dV}>0.$$The $\Delta_2$-condition gives $\int_Xw(f)dV\leq L_w$ for a uniform constant $L_w\geq1$ depending only on $w$, $dV$, and an upper bound for $\|f\|_w$. Choose $a_w>0$ and $B_0>1$ such that $w(t)\geq a_wt h(\log t)$ for $t\geq e^{B_0}$.
Thanks to Skoda-Tian's integrability theorem, 
    we can choose $\alpha=\alpha(n,A\omega_X)/2$ so that $\|e^{-\alpha v}\|_2\leq C_\alpha$ uniformly, since $v\in\PSH(X,A\omega_X)$ and $\sup_Xv=0$. Put
\(
 K:=4\left(\frac{cL_w}{a_wv_-(\theta)}\right)^{1/n}.
\)
Choose $B>B_0$ so large that $Kh(B)^{-1/n}\leq1$, and define
\[
 \gamma(t):=\frac{K}{\alpha}
 \int_B^{B+\alpha t}h(s)^{-1/n}\,ds,\qquad t\geq0.
\]
The function $\gamma$ is increasing and concave, $\gamma(0)=0$,
$0<\gamma'\leq1$, and
\[
 0\leq\gamma(t)\leq
 M:=\frac{K}{\alpha}\int_B^{+\infty}h(s)^{-1/n}\,ds<+\infty.
\]
      We set $\psi:=-\gamma(-v)$, which is $\omega$-psh since
    \[\omega_\psi=\gamma'(-v)\omega_v+[1-\gamma'(-v)]\omega-\gamma''(-v)dv\wedge d^c v\geq \gamma'(-v)\omega_v. \]
   It follows that $$\omega^n_\psi\geq (\gamma'(-v))^n\omega_v^n.$$
   These inequalities are justified by smooth approximation of the concave function $\gamma$ and truncation: set $v_j=\max\{v,-j\}$ and $\psi_j=-\gamma(-v_j)$. Plurifine locality gives $\omega_{\psi_j}^n\geq\mathbf1_{\{v>-j\}}(\gamma'(-v))^n\omega_v^n$, and bounded decreasing Bedford--Taylor convergence gives the result as $j\to\infty$, with $\psi=-M$ on $\{v=-\infty\}$, which is not charged by $\omega_v^n$.
On $\{\log f<-\alpha v+B\}$, we have
\(\omega_\varphi^n\leq ce^B e^{-\alpha v}dV\). On the other hand, on $\{\log f\geq -\alpha v+B\}$, 
\[\omega_\psi^n\geq (\gamma'(-v))^n\omega_v^n\geq \left[\gamma'\left(\alpha^{-1}(\log f-B)\right)\right]^n c_1 cw(f)dV. \]
Thus on $\{\log f\geq -\alpha v+B\}$, 
{\[\omega_\psi^n\geq K^nc_1a_wcf\,dV\geq 4^{n}cfdV=4^{n}\omega_\varphi^n.\]} 
Therefore, we have
\[\omega_\varphi^n\leq 4^{-n}\omega_\psi^n+ ce^Be^{-\alpha v}dV. \]  

\medskip
\textit{Step 2.2. Conclusion.}
Set $g:=ce^Be^{-\alpha v}$. Set $V:=\int_XdV$ and $\eta:=\frac{v_-(\theta)}{2\cdot4^n}$. We choose $C>0$ uniformly so large that
\[\int_{\{\varphi<-C\}}gdV\leq\|g\|_2(\frac{\|\varphi\|_1}{C})^{1/2}\leq\eta.\]
This is possible since $\|\varphi\|_1$ is uniformly bounded. Set
\[\hat g:=\mathbf{1}_{\{\varphi<-C\}}g+\frac{\eta}{V}.\]
Then $\eta\leq \|\hat g\|_1\leq2\eta$ and
\[\left\|\frac{\hat g}{\|\hat g\|_1}\right\|_2\leq\eta^{-1}\bigl(\|g\|_2+\eta V^{-1/2}\bigr).\]
It follows from Theorem~\ref{thm: exist-hermit}
that there exist $m>0$ and $u\in\PSH(X,\omega)\cap L^\infty(X)$ which solve \[(\omega+\dc u)^n=m\frac{\hat g}{\|\hat g\|_1}dV,\quad\sup_X u=0. \]
    We observe that $m\geq v_-(\theta)$. Also $c$ is uniformly bounded from above.
    By Step 1, $u\geq -C_0$ where $C_0$ depends only on $A\omega_X$, $dV$, $n$, $v_-(\theta)^{-1}$ and $\|g\|_2$.
    {Moreover, $\frac{m}{\|\hat g\|_1}\geq \frac{v_-(\theta)}{2\eta}=4^n$, so
    \[\omega_u^n\geq4^n\mathbf{1}_{\{\varphi<-C\}}g\,dV.\]}
 Then on the set $\{\varphi<\frac{\psi+u}{2}-C\}\subset\{\varphi<-C\}$, we have
 \begin{align*}
     \omega_\varphi^n\leq 4^{-n}\om_\psi^n+ 4^{-n}\om_u^n\leq \frac{1}{2^n}\left({\omega}+\dc\frac{\psi+u}{2} \right)^n.
 \end{align*} Then by the domination principle {
 (Proposition~\ref{Domination principle-1})}, $\varphi\geq \frac{\psi+u}{2}-C$. 
This completes the proof.
\end{proof}

\begin{theorem}\label{thm: exist-noncollapsing}
    Let $\theta$ be a semipositive (1,1)-form such that $v_-(\theta)>0$. Let $0\leq f\in L^w(dV)$ be such that $\int_X fdV=1$, where the weight $w$ satisfies \hyperlink{conditionK}{Condition (K)}.
    Then
    \begin{enumerate}
        \item for any $\lambda>0$ there exists a unique $\varphi_\lambda\in\PSH(X,\theta)\cap L^\infty(X)$ such that 
        \begin{equation}
            (\theta+\dc\varphi_\lambda)^n=e^{\lambda\varphi_\lambda}fdV.
        \end{equation}
        \item there exist a unique constant $c=c(\theta,f)>0$ and a $\varphi\in\PSH(X,\theta)\cap L^\infty(X)$ such that $\sup_X\varphi=0$ and 
        \begin{equation}\label{eq: cmae-semi-noncollapsing}
            (\theta+\dc\varphi)^n=cfdV.
        \end{equation}
    \end{enumerate}
\end{theorem}
\begin{proof}
    (1) Set $\omega_j:=\theta+2^{-j}\omega_X$. By Theorem~\ref{thm: exist-hermit}, 
    there exists $\varphi_j\in\PSH(X,\omega_j)\cap L^\infty$ such that
    \[(\omega_j+\dc\varphi_j)^n=e^{\lambda\varphi_j}fdV. \]
    For $j<k$, we have
    \[e^{-\lambda\varphi_j}(\omega_j+\dc\varphi_j)^n\leq e^{-\lambda\varphi_k}(\omega_j+\dc\varphi_k)^n , \]
    it follows from Corollary~\ref{coro: domination-semipositive} that $\varphi_j\geq \varphi_k$. Thus, $j\mapsto \varphi_j$ decreases. We infer that 
    \[(\omega_j+\dc\varphi_j)^n\leq e^{\lambda\sup_X\varphi_1}fdV. \]
 {Set $M_j:=\sup_X\varphi_j$ and $a_j:=\int_Xe^{\lambda\varphi_j}f\,dV$. By Lemma~\ref{lem: monotonicity-volume},
    \[
    v_-(\theta)\leq v_-(\omega_j)\leq a_j\leq e^{\lambda M_j},
    \qquad \lambda^{-1}\log v_-(\theta)\leq M_j\leq M_1.
    \]
    The densities $F_j:=a_j^{-1}e^{\lambda\varphi_j}f$ satisfy
    \[
    \int_XF_j\,dV=1,\qquad
    \|F_j\|_w\leq v_-(\theta)^{-1}e^{\lambda M_1}\|f\|_w.
    \]
    Choose $A>0$ such that $\omega_1\leq A\omega_X$. Applying Theorem~\ref{thm: estimate-noncollapsing} to $\varphi_j-M_j$, whose Monge--Amp\`ere measure is $a_jF_jdV$, implies that} $\textrm{Osc}(\varphi_j)\leq C$ for a constant $C>0$ depending on $X$, $dV$, $n$, {$A$,} $v_-(\theta)$, $w$, $\|f\|_w$ and $\lambda\sup_X\varphi_1$. Thus, it follows that $\varphi_j$ decreases to some $\varphi_\lambda\in\PSH(X,\theta)$, which is bounded below by the previous estimate and bounded above by $\sup_X\varphi_1$;
therefore, it is uniformly bounded.
{By \eqref{eq:varying-background-convergence} and dominated convergence,
\[
(\theta+\dc\varphi_\lambda)^n
=\lim_{j\to\infty}(\omega_j+\dc\varphi_j)^n
=e^{\lambda\varphi_\lambda}f\,dV
\]
in the weak sense.}
The uniqueness follows immediately from Corollary~\ref{coro: domination-semipositive}.

 If $\varphi$ is a solution of \eqref{eq: cmae-semi-noncollapsing}, then we set $\psi=\varphi{-}\inf_X\varphi+\lambda^{-1}\log c$. Hence, we have
 \[(\theta+\dc\psi)^n=cfdV\leq e^{\lambda\psi}fdV \]
The domination principle implies $\varphi_\lambda\leq \psi$, therefore $\sup_X\varphi_\lambda$ depends only on 
$\|\varphi\|_\infty$ and a positive lower bound for $\lambda$. Thus, $\|\varphi_\lambda\|_\infty$ does not depend on an upper bound $\lambda$.

\medskip \noindent (2) By Theorem~\ref{thm: exist-hermit}, we can find $c_j>0$ and $\varphi_j\in\PSH(X,\omega_j)$ such that 
\[(\omega_j+\dc\varphi_j)^n=c_jfdV,\quad\sup_X\varphi_j=-1. \]
It is easy to see that $c_j\geq v_-(\theta)$ for all $j$. Corollary~\ref{coro: domination-semipositive} shows that $j\mapsto c_j$ decrease to some $c=c(\theta,f)\in [v_-(\theta),c_1]$. Moreover, we have
\[(\omega_j+\dc\varphi_j)^n\leq c_1fdV,\quad\sup_X\varphi_j=-1. \]
It follows from Theorem~\ref{thm: estimate-noncollapsing} that $\varphi_j\geq -C$ for $C> 0$ depending on $X$, $dV$, $n$, $v_-(\theta)$, $c_1$, $w$, and $\|f\|_w$, for all $j$. Up to extracting  a subsequence, Hartogs' lemma  yields that $\varphi_j$ converges in $L^1(dV)$ and a.e. to some $\varphi\in\PSH(X,\theta)\cap L^\infty$. 
We retain the notation $\omega_j,c_j,\varphi_j$ for the subsequences.
We define 
\[\psi_j^+:=(\sup_{k\geq j}\varphi_k)^*,\quad \psi_j^{-}=P_\theta(\inf_{k\geq j}\varphi_k). \]
Observe that $\psi_j^+$ is a decreasing sequence that converges to $\varphi$, and $\psi_j^-$ is an increasing sequence {whose upper semicontinuously regularized limit $\psi:=(\lim_j\psi_j^-)^*$ belongs to $\PSH(X,\theta)\cap L^\infty(X)$; the convergence holds outside a pluripolar set}. In particular, by construction,
we have $\varphi\geq \psi$. We are going to show that $\varphi=\psi$. {For $k\geq j$, $\varphi_k\in\PSH(X,\omega_j)$ and
\[(\omega_j+\dc\varphi_k)^n\geq c_kf\,dV.\]
Applying Proposition~\ref{lem: GL23-19}(1) to finite maxima and then taking the increasing limit gives
\[(\omega_j+\dc\psi_j^+)^n\geq\inf_{k\geq j}c_kf\,dV=cf\,dV,\]
since $c_k\searrow c$. Fix $\ell$. For $j\geq\ell$, $\omega_\ell\geq\omega_j$, so
\[(\omega_\ell+\dc\psi_j^+)^n\geq cf\,dV.\]
Letting first $j\to\infty$ by the Bedford--Taylor convergence theorem (\cite[Theorem 3.18]{guedj2017degenerate}), and then $\ell\to\infty$ by the local expansion of the Monge--Amp\`ere operator, we obtain}
\[(\theta+\dc\varphi)^n\geq cfdV. \]
On the other hand, writing $(\omega_j+\dc\varphi_j)^n=e^{\varphi_j-\varphi_j}c_jfdV$, {set
\[p_{j,N}:=P_\theta\left(\min_{j\leq k\leq N}\varphi_k\right).\]
The repeated application of Proposition~\ref{lem: GL23-19} (2) gives
\[(\theta+\dc p_{j,N})^n\leq e^{p_{j,N}}\max_{j\leq k\leq N}(c_ke^{-\varphi_k})f\,dV.\]
Since $\varphi_k$ is uniformly lower bounded, the envelopes are uniformly bounded, and $p_{j,N}\searrow\psi_j^-$. Letting $N\to\infty$ by the Bedford--Taylor convergence theorem (\cite[Theorem 3.18]{guedj2017degenerate}) and dominated convergence yields}
\[(\theta +\dc\psi_j^-)^n\leq e^{\psi_j^- -\inf_{k\geq j}{\varphi_k}}\sup_{k\geq j}c_k fdV. \]
{Since $\psi_j^-\nearrow\psi$ outside a pluripolar set, $\inf_{k\geq j}\varphi_k\to\varphi$ almost everywhere, and $\sup_{k\geq j}c_k\to c$, the Bedford--Taylor convergence theorem (\cite[Theorem 3.23]{guedj2017degenerate}) and the dominated convergence give}
\[(\theta+\dc \psi)^n\leq e^{\psi-\varphi}cfdV\leq e^{\psi-\varphi}(\theta+\dc\varphi)^n. \]
Corollary~\ref{coro: domination-semipositive} yields $\varphi=\psi$. Therefore, $(\theta+\dc\varphi)^n=cfdV$.
{Hartogs' lemma gives $\sup_X\varphi=-1$; replacing $\varphi$ by $\varphi+1$ yields the required normalization. The uniqueness of $c$ follows from Corollary~\ref{coro: domination-semipositive} (1).

Finally, let $(c,\varphi)$ be the normalized solution obtained in (2), and set
\[
\psi_-:=\varphi+\lambda^{-1}\log c,\qquad
\psi_+:=\varphi-\inf_X\varphi+\lambda^{-1}\log c.
\]
Then $(\theta+\dc\psi_\pm)^n=cf\,dV$ and
\[e^{\lambda\psi_-}f\leq cf\leq e^{\lambda\psi_+}f.\]
Corollary~\ref{coro: domination-semipositive}(2) gives $\psi_-\leq\varphi_\lambda\leq\psi_+$, and hence
\[\|\varphi_\lambda\|_\infty\leq\|\varphi\|_\infty+\lambda^{-1}|\log c|.\]
Since $v_-(\theta)\leq c\leq c_1$ and Lemma~\ref{lem: nomalizing-constant} shows that $c_1$ is bounded from above, this estimate is uniform for $\lambda$ bounded away from zero and does not require an upper bound for $\lambda$.}
\end{proof}

\section{Big forms}
In this section, we assume $\theta$ is a real smooth {\em big} (1,1) form, i.e., there exists a $\theta$-psh function $\rho_0$ with {\em analytic singularities} satisfying $$\theta+\dc\rho_0\geq \varepsilon_0\omega_X$$ in the sense of currents for some $\varepsilon_0>0$.
We observe that $\rho_0$ is locally bounded in an open Zariski subset $\Omega:=X\setminus\{\rho_0=-\infty\}$. By subtracting a positive constant, we may assume that $\sup_X\rho_0=-1$.

We recall that (see~\cite[Theorem 8.19]{dinew2019lectures}) there exist $\alpha>0$ such that 
\begin{equation}
    \int_Xe^{-\alpha u}dV\leq C_\alpha ,\quad\forall\, u\in\PSH_0(X,\theta),
\end{equation} for $C_\alpha>0$ depending on $\alpha$, where $\PSH_0(X,\theta)=\{u\in\PSH(X,\theta):\sup_Xu=0\}$.

We prove the following lemma, generalizing the one in~\cite[Lemma 4.5]{BoucksomGuedjLu2025-volume}.

\begin{lemma}\label{lem: key-big} Let $w_h(t)\asymp_{+\infty}th(\log t)$ be a weight satisfying \hyperlink{conditionK}{Condition (K)} and \hyperlink{conditionI}{Condition~(I)}.
    Let $\psi\in\PSH(X,\theta)$ such that $0\geq \psi\geq \rho_0$. Let $0\leq f\in L^{w_h}(X,dV)$ be such that $\|f\|_{L^{w_h}}\leq 1$. Then there exists $v\in\PSH(X,\theta)$ with $\psi-1\leq v\leq \psi$ such that
    \begin{equation}
        (\theta+\dc v)^n\geq m f dV
    \end{equation} on $\Omega$, where $m>0$ depends only on $X$, $\omega_X$, $\theta$, $w_h$, $C_\alpha$, and $\sup_X\rho_0$.
\end{lemma}


\begin{proof} If $f=0$ almost everywhere, take $v=\psi$. We may therefore assume that $\|f\|_{L^{w_h}}>0$. Let $k$ be the function constructed in Lemma \ref{lem: weak-density}.
Set
\(
g:=k(2-\rho_0)f.
\)
Since $k>0$, we have $\|g\|_{L^{w_{\widetilde h}}}>0$.
Lemma \ref{lem: weak-density} gives $g\in L^{w_{\widetilde h}}$ for some increasing function $\widetilde h$ with $$w_{\widetilde h}\asymp_{+\infty} t\widetilde h(\log t),\qquad\int_1^{+\infty}\frac{dt}{\widetilde h^{1/n}(t)}<+\infty,$$ and $\|g\|_{L^{w_{\widetilde h}}}\leq C_h$ for $C_h>0$ depending on $n$, $w_h$, and $C_\alpha$.
It follows from Lemma~\ref{lem: subsolution}  that there exists
$u\in \PSH(X,\varepsilon_0\omega_X)$ such that
$-1\leq u\leq 0$
and
\[
(\varepsilon_0\omega_X+dd^c u)^n\geq  m' \frac{g}{\|g\|_{L^{w_{\widetilde h}}}}\,dV,
\]
where $m'>0$ is uniformly bounded away from zero.
Set
$u_1:=\rho_0+u$.
Since
\(
\theta+dd^c\rho_0\geq \varepsilon_0\omega_X,
\)
we have
\[
\theta+dd^c u_1
=
\theta+dd^c\rho_0+dd^c u
\geq
\varepsilon_0\omega_X+dd^c u
\geq 0.
\]
Thus $u_1\in \PSH(X,\theta)$ and
\[
(\theta+dd^c u_1)^n\geq m g\,dV,
\] with $m=m'\|g\|_{{w_{\widetilde h}}}^{-1}$. 
As in the proof of \cite[Lemma 4.5]{BoucksomGuedjLu2025-volume}, we set
$$
v:=\psi-\gamma(\psi-u_1),$$ where $\gamma$ is a concave increasing function such that
\[\gamma'(t)=c_\gamma[k(1+t)]^{-1/n},\;\gamma(0)=0,\quad c_\gamma=\min\left\{\left( \int_1^\infty\frac{dt}{k(t)^{1/n}}\right)^{-1}, k(1)^{1/n}\right\}. \]
Observe that
\[
1
\leq
1+\psi-u_1
=
1-u+\psi-\rho_0
\leq
2-\rho_0,
\]
since $-1\leq u\leq 0$ and $\rho_0\leq \psi\leq 0$.
We compute
\begin{align*}
\theta+dd^c v
&=
\theta+dd^c\psi
-
\gamma'(\psi-u_1){dd^c(\psi-u_1)}     
-
\gamma''(\psi-u_1)d(\psi-u_1)\wedge d^c(\psi-u_1)
\\
&\geq
\left(
1-\frac{c_\gamma}{[ k(1+\psi-u_1)]^{1/n}} 
\right)
(\theta+dd^c\psi)
+
c_\gamma\frac{\theta+dd^c u_1}
     {[ k(1+\psi-u_1)]^{1/n}}
\\
&\geq
c_\gamma\frac{\theta+dd^c u_1}
     {[ k(2-\rho_0)]^{1/n}}.
\end{align*}
Therefore, $v\in \PSH(X,\theta)$ with $\psi-1\leq v\leq \psi$, and
\[
(\theta+dd^c v)^n
\geq
\frac{c_\gamma^n (\theta+dd^c u_1)^n}{k(2-\rho_0)}
\geq
\frac{c_\gamma^n m g}{k(2-\rho_0)}\,dV.
\]
This completes the proof.
\end{proof}

We establish the relative $L^\infty$-estimate.
\begin{theorem}\label{thm: estimate-big} Let $\lambda\geq 0$.
    Let $\theta$ be a real smooth big (1,1) form. Let $0\leq f\in L^w(X,dV)$, where $w$ is a weight satisfying \hyperlink{conditionK}{Condition (K)} and \hyperlink{conditionI}{Condition~(I)}.
    Let $\psi\in\PSH(X,\theta)$ such that  $0\geq \psi\geq \rho_0-O(1)$. If $\varphi\in\PSH(X,\theta) $ such that $\sup_X\varphi=0$ if $\lambda=0$, $\psi= \varphi+O(1)$ and 
    \[(\theta+\dc \varphi)^n\leq e^{\lambda\varphi} fdV \] on $\Omega$, then $\varphi\geq \psi- C$ for a constant $C>0$ only depending on $\omega_X$, $\theta$, $w$, $\lambda$, and an upper bound for $\|f\|_{w}$.
\end{theorem}
\begin{proof} {Assume $0\geq \psi\geq \rho_0$, otherwise we consider $\widehat\psi:=\max\{\psi,\rho_0\}$.}
We have $\|f\|_w>0$, since otherwise $(\theta+\dc\varphi)^n=0$, contradicting with Corollary~\ref{coro: domination-big} (1).

       \textbf{Case 1:} $\lambda>0$. 
    By Lemma~\ref{lem: key-big}, we can find $v\in\PSH(X,\theta)$ such that $\psi-1\leq v\leq \psi$ and
    \[(\theta+\dc v)^n\geq m\frac{f}{\|f\|_{w}}dV \] for $m>0$ depending on $\omega_X$, $\theta$, $p$, $C_\alpha$ and $\sup_X\rho_0$. Set $m'=m\|f\|_{w}^{-1}$.
    Since
    \[e^{-\lambda\varphi}(\theta+\dc\varphi)^n\leq m'^{-1}(\theta+\dc v)^n\leq e^{-\lambda (v+\lambda^{-1}\log m')}(\theta+\dc v)^n, \]
    Corollary~\ref{coro: domination-big} yields $\varphi\geq v+\lambda^{-1}\ln m'\geq \psi-1+\lambda^{-1}\ln m'$.

    \medskip \textbf{Case 2: $\lambda=0$}. We normalize $\sup_X\varphi=0$.
    The dominated convergence theorem implies that
    \begin{equation*}
        \lim_{t\to+\infty}\int_Xw(\mathbf{1}_{\{f>t\}}f)dV=\int_X\lim_{t\to+\infty}w(\mathbf{1}_{\{f>t\}}f)dV=0.
    \end{equation*} By the $\Delta_2$-condition, this implies
\(
\|\mathbf1_{\{f>t\}}f\|_w\xrightarrow{t\to+\infty}0.
\)
     Hence, there is $B>0$ so large that \[\|\mathbf{1}_{\{f>B\}}f\|_{w}\leq 2^{-n-1}m, \]
where $m>0$ is a common choice of the constants in Lemma~\ref{lem: key-big} for the weights $w$ and $t^2$. By Lemma~\ref{lem: key-big} again, we can find $\phi\in\PSH(X,\theta)$ such that $\psi -1\leq \phi\leq \psi$ and
\[(\theta+\dc\phi)^n\geq 2^{n+1}\mathbf{1}_{\{f>B\}}f\,dV. \]   
We infer that
\[(\theta+\dc\varphi)^n\leq 2^{-n-1}(\theta+\dc\phi)^n+BdV. \]
Since normalized \(\theta\)-psh functions have a uniform \(L^1\) bound,
we can choose $C>0$ so large that $$\|\mathbf{1}_{\{\varphi<-C\}}\|_{2}\leq m\frac{1}{2^{n+1}B}.$$ By Lemma~\ref{lem: key-big}, there exists a function $v\in\PSH(X,\theta)$ such that $\psi-1\leq v\leq \psi$ and \[(\theta+\dc v)^n\geq m\frac{\mathbf{1}_{\{\varphi<-C\}}}{\|\mathbf{1}_{\{\varphi<-C\}}\|_{2}}dV\geq \mathbf{1}_{\{\varphi<-C\}}2^{n+1}BdV. \]
If $\{\varphi<-C\}$ has zero volume, the last measure inequality still holds.
 Then on the set $\{\varphi<\frac{\phi+v}{2}-C\}\subset\{\varphi<-C\}$, we have
 \begin{eqnarray*}
     (\theta+\dc\varphi)^n&\leq& \frac{1}{2^{n+1}}(\theta+\dc \phi)^n+ \frac{1}{2^{n+1}}(\theta+\dc v)^n\\
     &\leq& \frac{1}{2}\left(\theta+\dc\frac{\phi+v}{2} \right)^n.
 \end{eqnarray*} Then by the domination principle (Proposition~\ref{Domination principle-big}), $\varphi\geq \frac{\phi+v}{2}-C\geq \psi-1-C$.
This completes the proof.
\end{proof}

Recall that there exists $\rho_0\in\PSH(X,\theta)$ such that $\theta+\dc\rho_0\geq \varepsilon_0\omega_X$ for some $\varepsilon_0>0$. Set $\Omega:=\{\rho_0>-\infty\}$. 
\begin{lemma}\label{lem: approximate}
    There exists a sequence $(\psi_j)$ of $(\theta-4^{-j}\varepsilon_0\omega_X)$-psh functions with solvable singularities such that $0\geq \psi_j\geq \rho_0$ and \[{
\left(\lim_{j\to\infty}
P_\theta\!\left(\inf_{k\ge j}\psi_k\right)\right)^*
=V_\theta.}\]
\end{lemma}
\begin{proof}
    See~\cite[Theorem 4.6]{BoucksomGuedjLu2025-volume}. We approximate $V_\theta$ by a
 sequence $(\psi_j)$ of $\theta$-psh functions with solvable singularities. 
We first apply Demailly's approximation
 to obtain a decreasing sequence
$
\widetilde{\psi}_j
    \in \PSH(X,\theta+2^{-j}\varepsilon_0\omega_X)$
converging to $V_\theta$. Since the convergence is decreasing and $X$ is compact,
we have $\sup_X\widetilde{\psi}_j\to\sup_XV_\theta=0$.
After passing to a subsequence and relabeling, we may also assume that
$\sup_X\widetilde{\psi}_j\leq2^{-j}$. 
We then set
\[
\psi_j
    :=
    (1-2^{-j})\widetilde{\psi}_j+2^{-j}\rho_0.
\]
Since $\rho_0\leq-1$, we have
\[
\psi_j\leq(1-2^{-j})2^{-j}-2^{-j}=-4^{-j}<0.
\] 
A direct computation gives
\[\theta+\dc\psi_j\geq 4^{-j}\varepsilon_0\omega_X. \]
Since
$\psi_j
    \geq
    (1-2^{-j})V_\theta+2^{-j}\rho_0
    \geq \rho_0$,
 we have
\[
P_\theta^*
\left(
    \inf_{k\geq j}\psi_k
\right)
\geq
(1-2^{-j})V_\theta+2^{-j}\rho_0
\nearrow V_\theta,
\qquad\text{as }j\to+\infty,
\] completing the proof.
\end{proof}

\begin{theorem}\label{thm: exist-big1}
    Let $\theta$ be a big (1,1)-form. Let $0\leq f\in L^w(X,dV)$ be such that $\int_XfdV=1$, where $w$ is a weight satisfying \hyperlink{conditionK}{Condition (K)} and \hyperlink{conditionI}{Condition~(I)}.
    Then for any $\lambda>0$, there exists a unique $\varphi_\lambda\in\PSH(X,\theta)$ such that $\varphi_\lambda=V_\theta+O(1)$ and
        \begin{equation}
            (\theta+\dc\varphi_\lambda)^n=e^{\lambda\varphi_\lambda}fdV.
        \end{equation} Moreover, $|\varphi_\lambda-V_\theta|\leq C$, for a constant $C>0$ depending on $X$, $\theta$, $\lambda$, $w$ and upper bounds for $\|f\|_{w}$ and $\|f\|_{1/n}^{-1}$.
\end{theorem}
\begin{proof} The argument is similar to that of \cite[Theorem 4.6]{BoucksomGuedjLu2025-volume} with minor modifications.
 For simplicity, we treat the case $\lambda=1$. Set $\mu=fdV$.
 Let $(\psi_j)_{j\in\mathbb N}$ be a sequence of $\theta$-psh functions obtained in Lemma~\ref{lem: approximate}. 

By resolving the singularities of $\psi_j$ and applying
\cite[Proposition 3.12]{BoucksomGuedjLu2025-volume},
we obtain a modification
\(
\pi_j:Y_j\rightarrow X
\)
and a decomposition
\[
\pi_j^*(\theta+dd^c\psi_j)
=\omega_j+[E_j],
\]
where $\omega_j$ is a smooth semipositive and big form on $Y_j$ and
$E_j$ is an effective $\mathbb R$-divisor. 
Indeed,
\(
\omega_j\geq 4^{-j}\varepsilon_0\pi_j^*\omega_X\geq0.
\)
Moreover, $\omega_j$ is big by the exceptional-divisor construction
in the proof of~\cite[Proposition 3.12]{BoucksomGuedjLu2025-volume}.

We consider the measure
\(
\nu_j
:=e^{\psi_j\circ\pi_j}\pi_j^*\mu.
\)
Fix a smooth positive volume form $dV_{Y_j}$ on $Y_j$.
Writing
\(
\pi_j^*dV=J_j\,dV_{Y_j}
\),
we have
\[
\nu_j=g_j\,dV_{Y_j},
\qquad
g_j=e^{\psi_j\circ\pi_j}
(f\circ\pi_j)J_j.
\]
Here $J_j$ is the nonnegative Jacobian density
associated with the chosen volume forms. Since $\psi_j\leq0$, we have
\[
0\leq g_j\leq(f\circ\pi_j)J_j.
\]
Moreover, $J_j$ is smooth and bounded on $Y_j$.
By the change-of-variables formula and the
$\Delta_2$ condition on $w$, we obtain
\(
g_j\in L^w(Y_j,dV_{Y_j}).
\)

Fix $j$ and a Hermitian form $\eta_j$ on $Y_j$.
By Lemma~\ref{lem: key-big}, there exist $b_j\in\PSH(Y_j,\omega_j)$ and $m_j>0$ such that
\[
-1\leq b_j\leq0,\qquad
(\omega_j+\dc b_j)^n\geq m_jg_j\,dV_{Y_j}.
\]
For $0<\varepsilon\leq1$, Theorem~\ref{thm: exist-hermit} provides
$u_{j,\varepsilon}\in\PSH(Y_j,\omega_j+\varepsilon\eta_j)\cap L^\infty(Y_j)$ satisfying
\[
(\omega_j+\varepsilon\eta_j+\dc u_{j,\varepsilon})^n
=e^{u_{j,\varepsilon}}g_j\,dV_{Y_j}.
\]
Corollary~\ref{coro: domination-semipositive}(2) yields
\[
b_j+\log m_j\leq u_{j,\varepsilon}\leq u_{j,1},
\]
and shows that $u_{j,\varepsilon}$ decreases as $\varepsilon\to 0$.
Thus
$\widetilde\varphi_j:=\displaystyle\lim_{\varepsilon \to 0}u_{j,\varepsilon}
\in\PSH(Y_j,\omega_j)\cap L^\infty(Y_j)$.
By the Bedford--Taylor convergence theorem  (see~\eqref{eq:varying-background-convergence}), we obtain
\[
(\omega_j+\dc\widetilde\varphi_j)^n
=e^{\widetilde\varphi_j+\psi_j\circ\pi_j}\pi_j^*\mu.
\]
The function $\psi_j\circ\pi_j+
 \widetilde\varphi_j$ is $\pi^*\theta$-psh since
 \[ \pi^*_j\theta+\dc (\pi_j\circ\psi_j+
 \widetilde\varphi_j)=\omega_j+\dc\widetilde\varphi_j+[E_j]\geq 0.\]
It descends to some $\varphi_j\in\PSH(X,\theta)$ such that $\varphi_j=\psi_j+O(1)$ since $\widetilde\varphi_j$ is bounded; cf.~\cite[Proposition 1.2.7]{boucksomthese}. On the other hand, the non-pluripolar
Monge--Ampere measure does not charge the
exceptional set, and neither does
$\mu=f\,dV$. 
Therefore, we obtain
\begin{equation}\label{eq:corrected-downstairs}
(\theta+dd^c\varphi_j)^n
=e^{\varphi_j}f\,dV\quad\text{on } X.
\end{equation}

To see that $\varphi_j$ is uniformly bounded from above, we pick a Gauduchon metric $\omega_G$ and write
$\omega_G^n=G\,dV$.
The mixed Monge--Amp\`ere inequalities (\cite[Lemma 1.9]{nguyen2016complex}) yield
\[
(\theta+dd^c \varphi_j)\wedge\omega_G^{n-1}
\geq
e^{\sup_X \varphi_j/n}
e^{(\varphi_j-\sup_X\varphi_j)/n}
f^{1/n}G^{1-1/n}\,dV.
\]
Integrating over $X$ and using $
dd^c\omega_G^{n-1}=0$, we obtain
\[
e^{\sup_X\varphi_j/n}
\int_X
e^{(\varphi_j-\sup_X\varphi_j)/n}
f^{1/n}G^{1-1/n}\,dV
\leq C.
\]
{Since $w(t)/t\to+\infty$, bounded subsets of $L^w$ are
uniformly integrable. Together with the uniform $L^1$ bound for
normalized $\theta$-psh functions, this allows us to choose $R>0$,
depending only on the fixed data and upper bounds for $\|f\|_w$ and
$\|f\|_{1/n}^{-1}$, such that
\[
\begin{aligned}
\int_Xe^{(\varphi_j-\sup_X\varphi_j)/n}f^{1/n}G^{1-1/n}\,dV
&\geq\frac{e^{-R/n}}{2}\int_Xf^{1/n}G^{1-1/n}\,dV\\
&\geq c\,e^{-R/n}\|f\|_{1/n}^{1/n},
\end{aligned}
\]
where $c>0$ depends only on $G$.}
Therefore, $\sup_X\varphi_j$ is uniformly bounded. 

By Theorem~\ref{thm: estimate-big},  $\psi_j-C\le \varphi_j $ for $C>0$ uniformly.
{After passing to a subsequence, $\varphi_j\to\varphi$ in $L^1(dV)$ and almost everywhere, with $\varphi=V_\theta+O(1)$. Set
\[
u_j^+:=(\sup_{k\geq j}\varphi_k)^*,\qquad
u_j^-:=P_\theta(\inf_{k\geq j}\varphi_k),
\]
and $u:=(\lim_j u_j^-)^*$. These envelopes are locally uniformly bounded on $\Omega$, since $\rho_0-C\leq u_j^\pm\leq C$; moreover, Lemma~\ref{lem: approximate} gives $V_\theta-C\leq u\leq\varphi\leq V_\theta+C$.
The envelope argument in the proof of Theorem~\ref{thm: exist-noncollapsing} (2) yields
\[
(\theta+\dc\varphi)^n\geq e^\varphi f\,dV,\qquad
(\theta+\dc u)^n\leq e^u f\,dV.
\]
Corollary~\ref{coro: domination-big} gives $u=\varphi$, hence $(\theta+\dc\varphi)^n=e^\varphi f\,dV$.}
\end{proof}


\begin{theorem}\label{thm: exist-big0}
     Let $\theta$ be a big (1,1)-form.  Let $0\leq f\in L^w(X,dV)$ be such that $\int_XfdV=1$, where $w$ is a weight satisfying \hyperlink{conditionK}{Condition (K)} and \hyperlink{conditionI}{Condition~(I)}.
    Then there exist $c>0$ and $\varphi\in\PSH(X,\theta)$ such that $\varphi=V_\theta+O(1)$ and  
        \begin{equation}
            (\theta+\dc\varphi)^n=cfdV \qquad\text{on}\; \Omega, 
        \end{equation} where the constant $c>0$ is uniquely determined by $X$, $f$, and $\theta$. 
        Moreover, $|\varphi-V_\theta|\leq C$, for a constant $C>0$ depending on $X$, $\theta$, $w$, and upper bounds for $\|f\|_{w}$ and $\|f\|_{1/n}^{-1}$.
\end{theorem}

\begin{proof}
The uniqueness of the constant $c$ follows from Corollary~\ref{coro: domination-big}.

For each $j\in\mathbb N^*$, by Theorem~\ref{thm: exist-big1}, we can find
$\varphi_j\in\PSH(X,\theta)$
such that $\varphi_j=V_\theta+O(1)$ and
\[
(\theta+\dc \varphi_j)^n
= e^{j^{-1}\varphi_j}f\,dV.
\]
Set
$v_j:=\varphi_j-\sup_X\varphi_j$
and
$f_j:=e^{j^{-1}v_j}f$.
We rewrite the above equation as
\[
(\theta+dd^c v_j)^n
=
c_j f_j\,dV,
\qquad
c_j:=e^{j^{-1}\sup_X\varphi_j}.
\]
Extracting and relabeling, by Hartogs' lemma, we may assume that
$v_j\rightarrow v\in\PSH(X,\theta)$
in $L^1$ and a.e. Hence
$e^{j^{-1}v_j}
\rightarrow 1$
{in } $L^1(X,dV)$. In particular, $f_j\to f$ almost everywhere and in $L^1(dV)$.
Let $\omega_G$ be a Gauduchon metric on $X$. We write
$\omega_G^n=h\,dV$.
The mixed Monge--Amp\`ere inequality gives
\[
(\theta+dd^c \varphi_j)\wedge\omega_G^{n-1}
\geq
e^{j^{-1}v_j/n}
c_j^{1/n}
f^{1/n}h^{1-1/n}\,dV.
\]
Integrating this inequality over $X$ yields an upper bound for $c_j\leq M$. The Lemma~\ref{lem: key-big} yields a lower bound for $c_j\geq m$.
Extracting a subsequence, we can therefore assume that $c_j\rightarrow c>0$. Moreover, $0\leq c_jf_j\leq Mf$.
Thus, we have
\[ (\theta+\dc v_j)^n\leq MfdV,\quad\sup_X v_j=0.\]
Applying Theorem~\ref{thm: estimate-big} with $\psi=V_\theta$, we obtain $v_j\geq V_\theta-C$ for a uniform constant $C>0$. 
Set
\[
v_j^+:=(\sup_{k\geq j}v_k)^*,\qquad
v_j^-:=P_\theta(\inf_{k\geq j}v_k),
\]
and $v':=(\lim_j v_j^-)^*$. Since $V_\theta-C\leq v_j^\pm\leq V_\theta$, these envelopes are locally uniformly bounded on $\Omega$. As in the proof of Theorem~\ref{thm: exist-noncollapsing}(2), with Proposition~\ref{lem: BGL25-213}, we have
\[
(\theta+\dc v)^n\geq cf\,dV,\qquad
(\theta+\dc v')^n\leq e^{v'-v}cf\,dV
\quad\text{on }\Omega,
\]
using $c_jf_j\to cf$ almost everywhere and $0\leq c_jf_j\leq Mf$.
Since $V_\theta-C\leq v'\leq v\leq V_\theta$, Corollary~\ref{coro: domination-big} gives $v'=v$, hence $(\theta+\dc v)^n=cf\,dV$.
\end{proof}
\begin{theorem}
     Let $\theta$ be a big (1,1)-form.  Let $0\leq f\in L^w(X,dV)$ be such that $\int_XfdV=1$, where $w$ is a weight satisfying \hyperlink{conditionK}{Condition (K)}.
   If there exist $c>0$, $C_0>0$ and $\varphi\in\PSH(X,\theta)$ such that 
        \begin{equation}{\varphi\geq \rho_0-C_0},\quad\text{and}\;
            (\theta+\dc\varphi)^n=cfdV \quad\text{on}\; \Omega, 
        \end{equation} then $\varphi$ is continuous on $\Omega$.
\end{theorem}
\begin{proof}
    We first show that there exists a sequence of functions $0\geq \varphi_j\in\PSH(X,\theta)\cap\mathcal{C}^0(\Omega)$ such that $\varphi_j\searrow \varphi$. Indeed, let $h_j$ be a decreasing sequence of smooth functions that converges to $\varphi$.  Set $\varphi_j=P_\theta(h_j)$. Then $\varphi_j\geq \varphi$ decreases to some $\psi\geq \varphi$. On the other hand $\varphi_j\leq h_j$ for all $j\in\mathbb{N}$. Taking the limit implies $\psi\leq \varphi$, hence $\varphi=\psi$. For each $j$, it follows from Lemma~\ref{lem: eq-envelope} that
\[(\theta+\dc\varphi_j)^n=\mathbf{1}_{\{\varphi_j=h_j\}}(\theta+\dc h_j)_+^n  \]
which has $L^\infty$ densities. By~\cite[Theorem 1.1]{dang2026regularity}, $\varphi_j$ is H\"older continuous on $\Omega$.

\medskip

Fix $\lambda\in (0,1)$.  
        {Set $f_{j,\lambda}:=\lambda^{-n}\mathbf{1}_{\{\f<\f_j-\lambda\}}f$.
        For every $s>0$, the $\Delta_2$-condition and dominated convergence give
        \[
        \int_X w(f_{j,\lambda}/s)\,dV
        =\int_{\{\f<\f_j-\lambda\}}w(\lambda^{-n}f/s)\,dV\longrightarrow0.
        \]
        Hence $\|f_{j,\lambda}\|_w\to0$.
        By Lemma~\ref{lem: subsolution}, for all sufficiently large $j$ there exists $u_{j,\lambda}\in\PSH(X,\varepsilon_0\omega_X)$ such that
        \[
        -1\leq u_{j,\lambda}\leq0,\qquad
        (\varepsilon_0\omega_X+\dc u_{j,\lambda})^n\geq2c f_{j,\lambda}\,dV.
        \]}
        We consider
		\[v_{j,\lambda}:=(1-\lambda)\f_j +\lambda(\rho_0+u_{j,\lambda} ) -B\lambda,\]
		for  $B>0$ to be chosen hereafter. 
		By assumptions, we have $\f_j\geq \f\geq \rho_0-C_0$. Hence on $\{\f<v_{j,\lambda}\}$,
		\[\f<\f_j-\lambda(\f_j-\rho_0)+\lambda\sup_X|u_{j,\lambda}|-B\lambda\leq \f_j-\lambda, \] where we have chosen $B>2+C_0$. 
        We compute
		\begin{equation}\label{eq: cont1}
		(\theta+\dc v_{j,\lambda})^n\geq\lambda^n(\theta+\dc(\rho_0+u_{j,\lambda}))^n\geq {2c\mathbf{1}_{\{\f<\f_j-\lambda\}}f\,dV}.
		\end{equation}
        {It follows from~\eqref{eq: cont1} that}
		\[ {2}(\theta+\dc\f)^n\leq (\theta+\dc v_{j,\lambda})^n \]
		on $\{ \f<v_{j,\lambda}\}$. 
        By the domination principle (Proposition~\ref{Domination principle-big}), $\f\geq v_{j,\lambda}$ on $X$.
         Letting $j\to +\infty$ we obtain
		\[ \liminf_{j\to+\infty} \inf_K(\f-\f_j)\geq -(\sup_K|\rho_0|+1+B)\lambda,\]
		for any compact set $K\Subset \Omega$. Letting $\lambda\to 0$ we have $\f_j\to \f$ uniformly on $K$. This completes the proof.
\end{proof}
\bibliographystyle{alpha}
	\bibliography{bibfile}	
	
\end{document}